\documentclass[12pt, a4paper]{article}
\usepackage{float}
\usepackage{titlesec} 
\usepackage{geometry}
\usepackage{multirow}
\usepackage{enumerate}
\usepackage{booktabs}
\usepackage{ytableau}
\usepackage{tikz}
\usepackage{graphicx}
\usepackage{caption}
\usepackage{tikz}
\usepackage{mathrsfs}
\usepackage{makecell}
\usepackage{algorithm,algorithmic}
\usepackage{lmodern}
\usepackage{amssymb}
\usepackage{amsfonts}
\usepackage{graphicx}
\usepackage{epstopdf}
\usepackage{setspace}
\usepackage[unicode={ture},colorlinks,
            linkcolor=black,
            anchorcolor=blue,
            citecolor=green]{hyperref}        
\usepackage{appendix}
\usepackage{amsmath}
\usepackage{amsthm}

\usepackage{multirow}
\usepackage{colortbl}
\definecolor{darkblue}{rgb}{0.96, 0.96, 0.86}
\definecolor{darkgray}{rgb}{1.00, 0.97, 0.90}
\usepackage{makecell}
\usepackage{pifont}
\usepackage{authblk}

\usepackage{xcolor}
\usepackage[normalem]{ulem}

\allowdisplaybreaks[4]

\newtheorem{thm}{Theorem}[section]
\newtheorem{lem}[thm]{Lemma}

\newtheorem{cx}[thm]{Conjecture}
\newtheorem{pr}[thm]{Proposition}
\newtheorem{cor}[thm]{Corollary}

\theoremstyle{definition}

\newtheorem*{re}{Remark}

\title{  The second largest eigenvalue of some nonnormal Cayley graphs on symmetric groups }
\author{Yuxuan Li, Binzhou Xia and Sanming Zhou\\ 
School of Mathematics and Statistics, The University of Melbourne, Parkville, VIC 3010, Australia}
\date{}

\begin{document}
\maketitle
\openup 0.6 \jot 

\renewcommand{\thefootnote}{\empty}
\footnotetext{E-mail addresses: yuxuan11@student.unimelb.edu.au (Yuxuan Li), binzhoux@unimelb.edu.au (Binzhou Xia), sanming@unimelb.edu.au (Sanming Zhou)}

\begin{abstract}
A Cayley graph on the symmetric group $S_n$ is said to have the Aldous property if its strictly second largest eigenvalue (that is, the largest eigenvalue strictly smaller than the degree) is attained by the standard representation of $S_n$. For $1\leq r < k < n$, let $C(n,k;r)$ be the set of $k$-cycles of $S_n$ moving every point in $\{1, \ldots, r\}$. Recently, Siemons and Zalesski [J. Algebraic Combin. 55 (2022) 989--1005] posed a conjecture which is equivalent to saying that for any $n \ge 5$ and $1\leq r<k<n$ the nonnormal Cayley graph $\mathrm{Cay}(S_n, C(n,k;r))$ on $S_n$ with connection set $C(n,k;r)$ has the Aldous property. Solving this conjecture, we prove that all these graphs have the Aldous property except when (i) $(n, k, r) = (6, 5, 1)$ or (ii) $n$ is odd, $k = n-1$, and $1 \le r < \frac{n}{2}$. Along the way we determine all irreducible representations of $S_n$ that can achieve the strictly second largest eigenvalue of $\mathrm{Cay}(S_n, C(n,n-1;r))$ as well as the smallest eigenvalue of this graph.

\end{abstract}
\section{Introduction} 
\label{sec:Introduction}

We only consider finite simple undirected graphs in this paper. Suppose $\Gamma$ is such a graph and $A(\Gamma)$ is the adjacency matrix of $\Gamma$. Since $A(\Gamma)$ is real and symmetric, all its eigenvalues are real numbers, and they are called the \emph{eigenvalues} of $\Gamma$. We usually arrange them in a non-increasing manner as $\lambda_1\geq \lambda_2\geq \cdots\geq \lambda_n$ and denote the distinct ones among them by $\alpha_1> \alpha_2>\cdots> \alpha_m$. As in \cite{LXZ}, we refer to $\alpha_2$ as the \emph{strictly second largest eigenvalue} of $\Gamma$. Whenever we want to stress the dependence of the $i$-th largest eigenvalue (strictly $i$-th largest eigenvalue, respectively) of $\Gamma$ or a real symmetric matrix $M$, we write $\lambda_i (\Gamma)$ or $\lambda_i(M)$ ($\alpha_i(\Gamma)$ or $\alpha_i(M)$, respectively) in place of $\lambda_i$ ($\alpha_i$, respectively). The \emph{spectrum} of $M$, written $\mathrm{Spec}(M)$, is the collection of the eigenvalues of $M$ with multiplicities.  

Let $G$ be a finite group with identity element $\mathbf{1}$, and $S$ an inverse-closed subset of $G\setminus \{\mathbf{1}\}$. The \emph{Cayley~graph} on $G$ with respect to $S$, denoted by $\mathrm{Cay}(G, S)$, is the $|S|$-regular graph with vertex set $G$ and edge set $\{\{g,gs\}~|~g\in G, s\in S\}$. It is readily seen that $\mathrm{Cay}(G,S)$ is connected if and only if its \emph{connection set} $S$ is a generating subset of $G$. We say that $\mathrm{Cay}(G,S)$ is a \emph{normal} Cayley graph if $S$ is closed under conjugation; otherwise it is called a \emph{nonnormal} Cayley graph.

In \cite{A,AM,D,M}, it has been shown that the second largest eigenvalue of a regular graph reflects to some extent how well connected the overall graph is. An expander, roughly speaking, is a graph with small degree but strong connectivity properties. The second largest eigenvalue of Cayley graphs has been a focus of study for a long time as many important expanders are Cayley graphs. See \cite{HLW} for a survey on expander graphs with applications and \cite[Section 8]{LZ} for a collection of results on the second largest eigenvalue of Cayley graphs. The celebrated Aldous' spectral gap conjecture asserts that the second largest eigenvalue of any connected Cayley graph on the symmetric group $S_n$ with respect to a set of transpositions is attained by the standard representation of $S_n$. This conjecture in its general form was proved in \cite{CLR} after nearly twenty years of standing. In general, a Cayley graph $\mathrm{Cay}(S_n,S)$ on $S_n$ is said to have the \emph{Aldous property} \cite{LXZ} if its strictly second largest eigenvalue is attained by the standard representation of $S_n$. In \cite{LXZ}, the authors of the present paper identified three families of normal Cayley graphs on $S_n$ with the Aldous property, one of which can be considered as a generalization of the normal case of Aldous' spectral gap conjecture. Refer to \cite{HHC,PP} for further generalizations of Aldous' spectral gap conjecture for normal Cayley graphs on symmetric groups. 

One specific way to generalize Aldous' spectral gap conjecture is to verify whether the Cayley graph on a symmetric group with connection set formed by cycles of certain length has the Aldous property. Let $C(n,k)$ be the conjugacy class of all $k$-cycles in $S_n$ for $2\leq k\leq n$. With the help of the representation theory of symmetric groups, Siemons and Zalesski \cite{SZ} determined the strictly second largest eigenvalue of $\mathrm{Cay}(S_n, C(n,k))$ for $k=n$ or $n-1$. Their results indicate that when $k$ is $n$ or $n-1$, the normal Cayley graph $\mathrm{Cay}(S_n, C(n,k))$ does not possess the Aldous property. In \cite[Theorem 1.2]{MR}, Maleki and Razafimahatratra proved that if $n-k\ge 2$ is relatively small compared to $n$ then $\mathrm{Cay}(S_n, C(n,k))$ has the Aldous property. They conjectured \cite[Conjecture 1.4]{MR} that the same result holds for any $k$ between $2$ and $n-2$, that is, for any $n \ge 4$ and $2 \leq k \leq n-2$, the strictly second largest eigenvalue of $\mathrm{Cay}(S_n, C(n,k))$ is attained by the standard representation of $S_n$, and its value is $\frac{n-k-1}{n-1}\binom{n}{k}(k-1)!$. This conjecture was proved to be true for $k = 2$ \cite{DS}, $k=3$ \cite{HH} (see also \cite[Remark 4.1]{LXZ}), $k = 4$ \cite{HHC} and $k=5$ \cite{MR}. Recently, this conjecture was fully proved in \cite{LXZ2}, where the authors of the present paper discussed the strictly second largest eigenvalue of $\mathrm{Cay}(S_n,C(n,I))$ with $C(n,I)=\cup_{k\in I} C(n,k)$ and $I\subseteq \{2,3,\ldots,n\}$.

Determining the strictly second largest eigenvalue of nonnormal Cayley graphs is quite challenging in general
but has been settled for several families. For $1\leq i\leq j\leq n$, let $r_{i,j}\in S_n$ be the permutation which maps $i, i+1, \ldots, j-1, j$ to $j, j-1, \ldots, i+1, i$, respectively, and fixes all other points in $[n] = \{1, 2, \ldots, n\}$. The Cayley graphs on $S_n$ with connection sets $\{r_{1,j}~|~ 2\leq j\leq n\}$ and $\{r_{i,j}~|~1\leq i< j \leq n\}$ are called the \emph{pancake graph} $P_n$ and \emph{reversal graph} $R_n$, respectively. In \cite{C1}, Cesi determined the second largest eigenvalue of $P_n$ for $n\ge 3$, which was proved to be attained by the standard representation of $S_n$. Thus the pancake graphs have the Aldous property, which is called Property (A) in \cite{C1}.  Then Chung and Tobin in \cite{CT} determined the second largest eigenvalue of $R_n$ by recursively decomposing $R_n$ into $P_n$ and copies of $R_{n-1}$, where  one can tell from the proof that the reversal graphs also have the Aldous property. 

The method used in \cite{CT} for dealing with $R_n$ was further utilized by Huang and Huang \cite{HH} to determine the second largest eigenvalues of the Cayley graphs on the alternating group $A_n$ with connection sets $\{(1~2~i), (1~i~2)~|~ 3 \leq  i \leq  n\}$, $\{(1~ i~ j), (1~ j~ i)~|~ 2\leq i< j\leq n\}$ and $\{(i~ j~ k), (i~ k~ j)~|~1\leq i<  j<k\leq n\}$, respectively. Here $(i~j~k)$ denotes the $3$-cycle in  $A_n$ sending $i$ to $j$, $j$ to $k$ and $k$ to $i$. Suppose $r,k,n$ are three integers satisfying $1\leq r<k<n$. In \cite{SZ}, Siemons and Zalesski defined $H:=C(n,k;r)$, as a subset of $C(n,k)$, to be the set of $k$-cycles of $S_n$ moving every point from $1$ to $r$, and they determined all the eigenvalues of $H^+:=\sum_{h\in H}h\in \mathbb{C}S_n$ on the natural permutation module $M^{(n-1,1)}$ of $S_n$, among  which the second largest one is  $$\mu_2(n,k;r)=(k-2) !\binom{n-r}{k-r} \frac{1}{n-r}\left((k-1)(n-k)-\frac{(k-r-1)(k-r)}{n-r-1}\right).$$ Then $\mu_2(n,k;r)$ provides a lower bound for the strictly second largest eigenvalue of the nonnormal Cayley graph $\mathrm{Cay}(S_n,C(n,k;r))$ (see \cite[Theorem 1.3]{SZ}). Siemons and Zalesski \cite{SZ} then conjectured that $\mu_2(n,k;r)$ is exactly the strictly second largest eigenvalue of $\mathrm{Cay}(S_n,C(n,k;r))$ for all triplets $(n,k,r)$ satisfying $1 \leq r<k<n$. Since the natural permutation module $M^{(n-1,1)}$ decomposes into one trivial representation and one standard representation, $\mu_2(n,k;r)$  actually is the largest eigenvalue of $H^+$ on the standard representation of $S_n$.    Thus their conjecture can be restated in the following form.
\begin{cx}[\cite{SZ}] \label{cx:the conjecture}
	Suppose $n\ge 5$ and $1 \leq r<k<n$. The strictly second largest eigenvalue of $\mathrm{Cay}(S_n, C(n,k;r))$ is $\mu_2(n,k;r)$, attained by the standard representation of $S_n$. In other words, $\mathrm{Cay}(S_n, C(n,k;r))$ has the Aldous property.
\end{cx}

The case $k=r+1$ with $2 \leq r\leq n-2$ of this conjecture has been confirmed in \cite[Theorem 1.4]{SZ}. A recent result in \cite{HH} also indicates that $\alpha_2(\mathrm{Cay}(S_n,C(n, 3 ; 2)))=\mu_2(n,3;2)$. Combining these with the very early work in \cite{FOW} about $\lambda_2(\mathrm{Cay}(S_n,C(n, 2; 1)))=n-2$,  we have the following theorem. 

\begin{thm}{\rm  (\cite{FOW}, \cite{HH} and \cite[Theorem 1.4]{SZ})}\label{thm:induction_base}
	Suppose $n\ge 5$ and $2 \leq k \leq n-1$. The strictly second largest eigenvalue of $\mathrm{Cay}(S_n, C(n, k ; k-1))$ is $(k-1) !(n-k)$, attained by the standard representation of $S_n$. In other words, $\mathrm{Cay}(S_n, C(n, k ; k-1))$ has the Aldous property.
\end{thm}

Further evidence supporting Conjecture~\ref{cx:the conjecture} can be found in \cite{HH}, where one can deduce that $\alpha_2(\mathrm{Cay}(S_n, C(n,3;1)))=n^2-5n+5$, which is equal to $\mu_2(n,3;1)$. Combining this with Theorem \ref{thm:induction_base}, we see that Conjecture~\ref{cx:the conjecture} is true for $k = 2, 3$. In this paper we solve Conjecture~\ref{cx:the conjecture} in the general case where $4 \leq k \leq n-1$. We prove this conjecture is almost always true and determine exactly when the conjecture is not true. The following is a summary of our main results.

\begin{thm}\label{thm:main}
Suppose $n \ge 5$, $k\ge 4$ and $1\leq r<k\leq n-1$. Then $\mathrm{Cay}(S_n,C(n,k;r))$ has the Aldous property except for the following cases:
\begin{itemize}
	\item[\rm (1)]  $n=6$, $k=5$ and $r=1$;
	\item[\rm (2)]  $n\ge 5$ is odd, $k=n-1$ and $1\leq r<\frac{n}{2}$.  
\end{itemize}  
\end{thm}

The overall method used in this paper is induction. One induction base is the case $r=k-1$ and the Aldous property of $\mathrm{Cay}(S_n,C(n,k;k-1))$ as seen in  Theorem~\ref{thm:induction_base}.  We then build another induction base in Section~\ref{sec:proof_for_theorem_n-1}, dealing with the case $k=n-1$. As we can see from Theorem~\ref{thm:main}, there are some exceptions in the case $k=n-1$. Thus in Section~\ref{sec:n-2} we first build some additional induction base in the case $k=n-2$ and then we do induction on both $n-k$ and $n-r$ to show that $\mathrm{Cay}(S_n,C(n,k;r))$ has the Aldous property whenever $4\leq k\leq n-2$. Here this induction method is similar to that used in \cite{CT} for dealing with the reversal graphs. We decompose $\mathrm{Cay}(S_n,C(n,k;r))$ into one $\mathrm{Cay}(S_n,C(n,k;r+1))$ and $n$ copies of $\mathrm{Cay}(S_{n-1},C(n-1,k;r))$.   Thanks to a recurrence relation for $\mu_2(n,k;r)$ (see \eqref{eq:recurrence}), we can finally prove in Theorems~\ref{thm:even_n-2} and \ref{thm:odd_n-2} that $\alpha_2(\mathrm{Cay}(S_n,C(n,k;r)))=\mu_2(n,k;r)$ for any $4\leq k\leq n-2$ and $1 \le r < k$. 

The most subtle case $k = n - 1$ will be handled in Section~\ref{sec:proof_for_theorem_n-1}, where the approach used is to decompose the nonnormal connection set $C(n,n-1;r)$ into some mutually disjoint subsets $P_i:=C(n,n-1)\cap G_{i}$ with $r+1\leq i \leq n$. Here $G_{i}$ is the stabilizer of $i\in [n]$ in $S_n$. This kind of decomposition works well when $r$ is large. For small $r$, we decompose  $C(n,n-1;r)$ in another way by deleting $P_i:=C(n,n-1)\cap G_{i}$ with $1\leq i\leq r$ from $ C(n,n-1)$. Now each $P_i$ is normal not in $S_n$ but in some subgroup of $S_n$ which is isomorphic to $S_{n-1}$. With these decompositions, we can apply Branching Rule  and use irreducible characters of $S_{n-1}$ to express the eigenvalues of each irreducible representation of $S_n$ on every $P_i^+ =\sum_{h\in P_i} h \in \mathbb{C}S_n$. Then by Weyl Inequalities, we can make use of these eigenvalues to bound that of  each irreducible representation of $S_n$ on $C(n,n-1;r)$ and finally identify the strictly second largest eigenvalue of $\mathrm{Cay}(S_n,C(n,n-1;r))$. 

The main results in Section~\ref{sec:proof_for_theorem_n-1} are Theorems~\ref{thm:n-1_odd} and \ref{thm:n-1_even}, which are summarized in the following Table~\ref{tab:main}. The third column of this table shows the partitions of $n$ with their corresponding irreducible representations attaining the strictly second largest eigenvalue  of $\mathrm{Cay}(S_n,C(n,n-1;r))$ and the last column demonstrates the multiplicity of that eigenvalue. Note that the standard representation of $S_n$ just corresponds to the partiton $(n-1,1)$. Thus from Table~\ref{tab:main} we know exactly when $\mathrm{Cay}(S_n,C(n,n-1;r))$ has the Aldous property.  As byproducts of Theorems~\ref{thm:n-1_odd} and \ref{thm:n-1_even}, we also determined the (second) smallest eigenvalue of $\mathrm{Cay}(S_n,C(n,n-1;r))$ in Corollaries~\ref{cor:odd_small} and \ref{thm:even_small}.

Theorem \ref{thm:main} will follow from Theorems \ref{thm:n-1_odd}, \ref{thm:n-1_even}, \ref{thm:even_n-2} and \ref{thm:odd_n-2} with no effort.

\begin{table}[thp]
\renewcommand\arraystretch{2.0}
\centering
\begin{tabular}{|c|c| c | c|}
\hline
\multicolumn{2}{|c|}{~~~~~\bf{$\mathbf{C(n,n-1;r)}$ with $\mathbf{n\ge 7}$}~~~~~ } & ~~~~\bf{Partitions}~~~~ & ~~~~\bf{Multiplicity}~~~~  \\\hline
\multirow{4}*{$n$ is odd} 
& $r=1$ & $(2,1^{n-2})$ & $n-1$\\ \cline{2-4}
&$2\leq r<\frac{n}{2}$ & $(2,1^{n-2})$   & $(n-1)(r-1)$\\ \cline{2-4}
& $\frac{n}{2}<r\leq n-2$ & $(n-1,1)$  & $(n-1)(n-r-1)$\\ \hline
\multirow{3}*{$n$ is even} 
& \multirow{2}*{$r=1,2$} & $(n-1,1),~(2,1^{n-2})$ & \multirow{2}*{Unknown} \\ 
&   & ($(n-2,1^2)$, $(3,1^{n-3})$ ) & \\ \cline{2-4}
& $3\leq r\leq n-2$ & $(n-1,1),~(2,1^{n-2})$  & $2(n-1)(n-r-1)$ \\ \hline
\end{tabular}
\caption{Summary of Theorems \ref{thm:n-1_odd} and \ref{thm:n-1_even}.  The partitions enclosed in parentheses are the potential candidates for attaining the strictly second largest eigenvalue. }
\label{tab:main}
\end{table}

\section{Preliminaries} 
\label{sec:preliminary}

The representation theory of finite groups plays a critical role in studying eigenvalues of Cayley graphs, as will be seen shortly in Propositions \ref{prop1.1} and \ref{prop1.2}. For basic concepts and properties of representations and characters of groups, the reader is referred to \cite{I,JL,S1,S}. In what follows, for any finite group $G$ we use 
$$
\widehat{G}=\{\rho_1,\rho_2,\ldots,\rho_k\}
$$ 
to denote a complete set of inequivalent (complex) irreducible matrix representations of $G$, with the assumption that $\rho_1$ is the trivial representation. For any $\rho_i \in \widehat{G}$, the map 
$$
\chi_i: g \mapsto \mathrm{Trace}(\rho_i(g)),\;\, g \in G
$$ 
is the \emph{character} of $\rho_i$, and the ratio 
$$
\tilde{\chi}_i(g):=\frac{\chi_i(g)}{\chi_i(\mathbf{1})}
$$ 
is known as the \emph{normalized character} of $\rho_i$ on $g\in G$, where $\chi_i(\mathbf{1})$ equals the dimension $\dim\rho_i$ of $\rho_i$. Note that $\dim \rho_1 = 1$ for the trivial representation $\rho_1$.

It is known \cite{MS} that the adjacency matrix of $\mathrm{Cay}(G, S)$ equals $\sum_{s\in S} R_{\mathrm{reg}}(s)$, where $\mathrm{reg}$ is the right regular representation of $G$ and $R_{\mathrm{reg}}(s)$ is the permutation matrix depicting the multiplication on $G$ from the right by $s$. The regular representation of $G$ decomposes as a direct sum of all irreducible representations of $G$, each appearing with multiplicity identical to its dimension. Therefore, we have the following proposition, where
$$
\rho_i(S) := \sum_{s \in S} \rho_i(s)
$$
and $\oplus$ denotes the direct sum of matrices. 

\begin{pr}{\rm \cite[Proposition 7.1]{MS}}\label{prop1.1}
The adjacency matrix of $\mathrm{Cay}(G,S)$ is similar to 
\begin{equation*}
	d_1 \rho_1(S)\oplus d_2\rho_2(S)\oplus \cdots\oplus d_k\rho_k(S),
\end{equation*}
where $d_i$ is the dimension of $\rho_i\in \widehat{G}$ and $$d_i\rho_i(S)=\underbrace{\rho_i(S)\oplus\rho_i(S)\oplus \cdots \oplus \rho_i(S)}_{d_i}.$$
\end{pr}

This implies that the multiset of eigenvalues of $\mathrm{Cay}(G,S)$ is the union of $d_i$ multisets of eigenvalues of $\rho_i(S)$ for $1\leq i\leq k$. If a fixed number $\lambda$ is an eigenvalue of each $\rho_i(S)$ with multiplicity $m_i$, which could be $0$, then as an eigenvalue of $\mathrm{Cay}(G,S)$ its multiplicity equals $\sum_{i=1}^k d_i\cdot m_i$. 
In the case that $\mathrm{Cay}(G,S)$ is normal, by Schur's Lemma, all $\rho_i(S)$'s are scalar matrices (see \cite[Lemma 5]{DS}) and the eigenvalues of $\mathrm{Cay}(G,S)$ can be expressed in terms of the irreducible characters of $G$ in the following way. 

\begin{pr}[\cite{DS,Z}]\label{prop1.2}
Let $\{\chi_1,\chi_2,\ldots,\chi_k\}$ be a complete set of inequivalent irreducible characters of $G$. Then the eigenvalues of any normal Cayley graph $\mathrm{Cay}(G, S)$ on $G$ are given by 
\begin{equation*}
	\lambda_j=\frac{1}{\chi_j(\mathbf{1})}\sum_{s\in S}\chi_j(s)=\sum_{s\in S} \tilde{\chi}_j(s),\quad j=1,2,\ldots,k.
\end{equation*}
Moreover, the multiplicity of $\lambda_j$ is equal to $\sum_{1\leq i\leq k,~\lambda_i=\lambda_j}\chi_i(\mathbf{1})^2.$
\end{pr}

We say that the strictly second largest eigenvalue of $\mathrm{Cay}(G,S)$ is \emph{attained} or \emph{achieved} by $\rho_i \in \widehat{G}$ if
$$\alpha_2(\mathrm{Cay}(G,S)) \in \mathrm{Spec}(\rho_i(S));$$
when $\mathrm{Cay}(G,S)$ is normal, this is equivalent to 
$$
\alpha_2(\mathrm{Cay}(G,S))=\sum_{s\in S} \tilde{\chi}_i(s).
$$
Let $c=[G:\langle S\rangle ]$ be the index of the subgroup $\langle S\rangle$ in $G$. Then $\mathrm{Cay}(G, S)$
is the union of $c$ copies of the connected Cayley graph $\mathrm{Cay}(\langle S\rangle, S)$ with degree $|S|$. Thus the largest eigenvalue $|S|$ of $\mathrm{Cay}(G, S)$ has multiplicity $c$, and the strictly second largest eigenvalue $\alpha_2(\mathrm{Cay}(G,S))$ just equals $\lambda_{c+1}(\mathrm{Cay}(G,S))$.

A \emph{partition} of a positive integer $n$ is a sequence of positive integers $\gamma=(\gamma_1,\gamma_2,\ldots,\gamma_m)$ satisfying $\gamma_1\ge \gamma_2\ge\cdots\ge\gamma_m$ and $n=\gamma_1+\gamma_2+\cdots+\gamma_m$.  We use $\gamma \vdash n$ to indicate that $\gamma$ is a partition of $n$. We say $\gamma$ is a \emph{hook} if $\gamma=(n-m,1^m)$ with $0\leq m\leq n-1$, and $\gamma=(n-m,2,1^{m-2})$ with $2\leq m\leq n-2$ is called a \emph{near hook}. 
 Recall from the representation theory of symmetric groups \cite{S} that for each partition of $n$, we can construct an irreducible module of $S_n$ known as the Specht module. It is well known that all the Specht modules corresponding to partitions of $n$  form a complete list  $\widehat{S_n}$ of inequivalent irreducible modules of $S_n$. For any subset $H$ of $S_n$, define
$$
H^{+}=\sum_{h\in H} h \in \mathbb{C}S_n.
$$

\begin{lem}{\rm \cite[Lemma 6.3]{SZ}}\label{lem:symmetric}
 	Let $G=S_n$ and let $L$ be an irreducible $G$-module. Then there is a basis of $L$ such that the matrices of $g+g^{-1}$ on $L$ are symmetric for any $g \in G$. Consequently, if $H \subset G$ is a subset such that $H=H^{-1}$, then the matrix of $H^{+}$ on $L$ is symmetric.
 \end{lem} 

For $\zeta\vdash n$, we use $S^\zeta$ to denote the corresponding Specht module of $S_n$ and $\rho_\zeta$ the matrix representation of $S^\zeta$, with dimension $d_\zeta$, under the basis such that $\rho_\zeta(H)$ is symmetric and thus has real eigenvalues whenever $H\subset S_n$ is closed under inverse. The existence of this basis is guaranteed by Lemma~\ref{lem:symmetric}. The representations $\rho_{(n)}$ and $\rho_{(1^n)}$ are just the \emph{trivial} and the \emph{sign representations} of $S_n$, respectively. The \emph{standard representation} of $S_n$ refers to $\rho_{(n-1,1)}$. Sometimes we say $\zeta$, instead of $\rho_\zeta$, attains the strictly second largest eigenvalue of  some Cayley graph on $S_n$.  Let $\chi_\zeta(\cdot)$ and $\tilde{\chi}_\zeta(\cdot)$ denote the character and the normalized character of $\rho_\zeta$, respectively.  We have $\chi_{(n)}(\sigma)=\tilde{\chi}_{(n)}(\sigma)= 1$ and $\chi_{(1^n)}(\sigma)=\tilde{\chi}_{(1^n)}(\sigma)=\mathrm{sgn}(\sigma)$ for any $\sigma\in S_n$.

 A \emph{Young~diagram} is a finite collection of boxes arranged in left-justified rows, with the row sizes weakly decreasing. The Young diagram associated to the partition $\zeta=(\zeta_1,\zeta_2,\ldots,\zeta_m)$ is the one that has $m$ rows and $\zeta_i$ boxes on the $i$-th row. The notation $\zeta^-$ denotes any partition of $n-1$ whose Young diagram is obtained by removing one box from that of $\zeta$. Denote by $a=(i, j)$ the box in the $i$-th row and $j$-th column of the Young diagram of $\zeta$. Then it has \emph{hook}
$$ 
H_{a}=H_{i, j}=\left\{\left(i, j^{\prime}\right)\in \zeta: j^{\prime} \geq j\right\} \cup\left\{\left(i^{\prime}, j\right)\in \zeta: i^{\prime} \geq i\right\}
$$
with corresponding \emph{hook-length}
$$
h_{a}=h_{i, j}=\left|H_{i, j}\right|.
$$
To illustrate, if $\zeta=\left(4^{2}, 3^{3}, 1\right)$, then the dotted boxes in
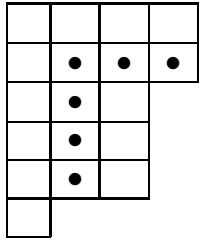
\begin{table}[H]
\centering
\begin{tabular}{|l|lll}
\hline
$\ $ & \multicolumn{1}{l|}{} & \multicolumn{1}{l|}{} & \multicolumn{1}{l|}{} \\ \hline
 & \multicolumn{1}{l|}{$\bullet$} & \multicolumn{1}{l|}{$\bullet$} & \multicolumn{1}{l|}{$\bullet$} \\ \hline
 & \multicolumn{1}{l|}{$\bullet$} & \multicolumn{1}{l|}{} &                       \\ \cline{1-3}
 & \multicolumn{1}{l|}{$\bullet$} & \multicolumn{1}{l|}{} &                       \\ \cline{1-3}
 & \multicolumn{1}{l|}{$\bullet$} & \multicolumn{1}{l|}{} &                       \\ \cline{1-3}
 &                       &                       &                       \\ \cline{1-1}
\end{tabular}
\end{table} 
\noindent  are the hook $H_{2,2}$ with hook-length $h_{2,2}=6$. The following lemma states the well-known Hook Formula for the dimension $\chi_\zeta(\mathbf{1})$ of any $\rho_\zeta\in \widehat{S_n}$, where $\mathbf{1}$ is the identity element of $S_n$.

\begin{lem}{\upshape \cite[Theorem 3.10.2]{S} }\label{lem:Hook} If $\zeta\vdash n$, then
	\[\chi_\zeta(\mathbf{1})=\frac{n !}{\prod_{(i, j) \in \zeta} h_{i, j}}.\]

\end{lem}

Branching Rule describes how irreducible representations of $S_n$ decompose into irreducible representations of $S_{n-1}$.

\begin{thm}{\rm(Branching Rule, \cite[Theorem 2.8.3]{S})}
  	If $\zeta \vdash n$, then  $S^\zeta \downarrow_{S_{n-1}} \cong \bigoplus_{\zeta^{-}} S^{\zeta^{-}}. $
  	
  \end{thm}

The \emph{conjugate} or \emph{transpose} of a partition $\zeta=(\zeta_1,\zeta_2,\ldots,\zeta_m)\vdash n$ is defined as $\zeta'=(\zeta_1',~\zeta_2',~\ldots,\zeta_h') \vdash n$, where $\zeta_i'$ is the length of the $i$-th column of $\zeta$. In other words, the Young diagram of $\zeta'$ is just the transpose of that of $\zeta$. 
The relation between $\rho_\zeta(\cdot)$ and $\rho_{\zeta'}(\cdot)$ is reflected in the following lemma.

\begin{lem}{\upshape (\cite[2.1.8]{JK}) }\label{lem:transpose}
  For any $\zeta \vdash n$, we have 
  \begin{equation*}
	\rho_{\zeta'}(\cdot)=\mathrm{sgn}(\cdot)\rho_\zeta(\cdot).
  \end{equation*} 
\end{lem}

The next two lemmas formulate the normalized  characters of $S_n$ on $n$-cycles and $(n-1)$-cycles, respectively, which follow from \cite[Lemma 4.10.3]{S}, \cite[Lemma 4.3]{SZ} and Hook Formula.  

\begin{lem}{\rm \cite[Lemma 4.10.3]{S}}\label{lem:n-cycle}
 	Suppose $\zeta\vdash n$. If $\sigma$ is an $n$-cycle of $S_n$, then
$$
\tilde{\chi}_\zeta(\sigma)= \begin{cases}\frac{(-1)^m n(n-m-1)!m!}{n!} , & \text { if } \zeta=\left(n-m, 1^m\right) \text { with } 0 \leq m \leq n-1 \\ 0, & \text { otherwise. }\end{cases}
$$
\end{lem}

\begin{lem}{\upshape \cite[Lemma 4.3]{SZ}} \label{lem:n-1_cycles}
	Suppose $\zeta\vdash n$. If $\sigma$ is an $(n-1)-$cycle of $S_n$, then 
	\[\tilde{\chi}_\zeta(\sigma)=\begin{cases}
	    1,& \text{if}~\zeta=(n);\\
	    (-1)^{n-2},&\text{if}~\zeta=(1^n);\\
		\frac{(-1)^{m-1}(n-1)(n-m)(n-m-2)!m(m-2)!}{n!},& \text{if}~\zeta=(n-m,2,1^{m-2})~\text{with}~2\leq m\leq n-2;\\
		0,& \text{otherwise}.
	\end{cases}\]
\end{lem}

Upper and lower bounds on eigenvalues of a symmetric matrix over the reals can be obtained from a theorem of Hermann Weyl \cite{W}. See also \cite[Theorem 2.8.1]{BH}. 

\begin{thm}{\rm (Weyl Inequalities)}
Let $C=A+B$ be symmetric $m \times m$ matrices, and let $\gamma_1 \geq \cdots \geq \gamma_m, \alpha_1 \geq \cdots \geq \alpha_m, \beta_1 \geq \cdots \geq \beta_m$ be the eigenvalues of $C, B, A$, respectively. Then for $i, j \in \{1,2,\ldots,n\}$, we have $\gamma_{i+j-1} \leq \alpha_i+\beta_j \text { whenever } i+j-1 \leq m$ and $\gamma_{i+j-m} \geq \alpha_i+\beta_j \text { whenever } i+j-1 \geq m $. In particular,
$$\gamma_1\leq \alpha_1+\beta_1\quad \text{and}\quad \gamma_m\ge \alpha_m+\beta_m.  $$
 \end{thm} 

 \begin{re}
 	In this paper, we may apply Weyl Inequalities to symmetric matrices $A,B,C$ satisfying $C=A-B$. In this situation, the eigenvalues $\{\gamma_i\}_{i=1}^m$ of $C$ are bounded by the eigenvalues $\{\alpha_i\}_{i=1}^m$ of $A$ and the eigenvalues $\{\beta_i\}_{i=1}^m$ of $B$ in the following way:
 	\begin{align*}
 		\gamma_{i+j-1} \leq \alpha_i -\beta_{m-j+1} \text{ whenever } i+j-1\leq m, \\
 		\gamma_{i+j-m} \geq \alpha_i -\beta_{m-j+1} \text{ whenever } i+j-1\geq m.
 	\end{align*}
 \end{re}

Recall that $G_i$ is the stabilizer of $i\in [n]$ in $S_n$. We use $G_{i,j}$ to denote the set of permutations of $S_n$ sending $i$ to $j$. Then all the right cosets of $G_i$,
\begin{equation}\label{eq:right_cosets_equitable_partition}
	\Pi_i:S_n=G_{1,i}\cup G_{2,i}\cup \cdots \cup G_{n,i}
\end{equation}
form a decomposition of $S_n$. In \cite{HHC} the edge set of a Cayley graph $\mathrm{Cay}(G,S)$ is defined as $\{\{g,sg\}~|~g\in G,~s\in S\}$, and when composing two permutations $\sigma\circ \tau$ in $S_n$, the authors in \cite{HHC} do it from left to right. However, in this paper we define the edge set of $\mathrm{Cay}(G,S)$ as $\{\{g,gs\}~|~g\in G,~s\in S\}$ and when composing two permutations $\sigma\circ \tau$ in $S_n$, we do it from right to left. Considering this nonessential difference, we apply \cite[Theorem 7]{HHC} to $\mathrm{Cay}(S_n,C(n,k;r))$  in the following manner. 

\begin{lem}{\rm \cite[Theorem 7]{HHC}}\label{lem:HHC_upper_bound_1}
	Let $H=C(n,k;r)$ and $\Gamma=\operatorname{Cay}(S_n, H)$, where $n \ge 5$ and $1 \le r < k < n$. The right coset decomposition $\Pi_i$ of $S_n$ given in \eqref{eq:right_cosets_equitable_partition} leads to an equitable partition of $\Gamma$, and the corresponding quotient matrix $\mathbf{B}$ is symmetric and independent on the choice of $i\in[n]$. Moreover, if $\lambda$ is an eigenvalue of $\Gamma$ other than that of $\mathbf{B}$, then, for each $j \in[n]$, we have
\begin{equation}\label{ineq:HH_upper_bpund}
	\lambda \leqslant \lambda_2\left(\operatorname{Cay}\left(G_j, H \cap G_j\right)\right)+\lambda_2\left(\operatorname{Cay}\left(S_n, H \backslash\left(H\cap G_j\right)\right)\right),
\end{equation}
where $G_j$ is the stabilizer of $j\in [n]$ in $S_n$.
\end{lem}

\begin{re}
\begin{enumerate}[{\rm (\romannumeral1)}]
	\item The quotient matrix $\mathbf{B}$ of $\Pi_i$ for any $i\in[n]$ is exactly the permutation matrix of $H^+$ arising from the natural permutation module $M^{(n-1,1)}$ of $S_n$. Thus as the second largest eigenvalue of $H^+$ on $M^{(n-1,1)}$, $\mu_2(n,k;r)$ is exactly $\lambda_2(\mathbf{B})$.  By Young's Rule, the natural permutation module $M^{(n-1,1)}$ of $S_n$ decomposes into one trivial module $S^{(n)}$  and one  $S^{(n-1,1)}$. Thus the spectrum of $H^+$ on $M^{(n-1,1)}$ is the union of the spectra of $H^+$ on $S^{(n)}$ and $S^{(n-1,1)}$. Clearly, $H^+$ acting on $S^{(n)}$ gives the largest eigenvalue $|H|$ of $\mathrm{Cay}(S_n,H)$. Then the spectrum of $H^+$ on  $S^{(n-1,1)}$ is obtained by dividing the largest eigenvalue $|H|$ from the spectrum of $H^+$ on $M^{(n-1,1)}$. In particular, the second largest eigenvalue $\mu_2(n,k;r)$ of $H^+$ on $M^{(n-1,1)}$, which is also $\lambda_2(\mathbf{B})$, is exactly the largest eigenvalue of $H^+$ on  $S^{(n-1,1)}$, that is, $\alpha_1(\rho_{(n-1,1)}(H))$.

	\item When $j\leq r$, the definition of $H=C(n,k;r)$ implies that $H\cap G_j=\emptyset$, and thus the right hand side of \eqref{ineq:HH_upper_bpund} is just $\lambda_2(\mathrm{Cay}(S_n,C(n,k;r)))$, which is a trivial upper bound for $\lambda$. When making use of \eqref{ineq:HH_upper_bpund}, we shall take $j$ from $[n]\setminus [r]$ and mostly we just let $j=n$.
\end{enumerate}
\end{re}

 The following two lemmas on the spectrum of $\rho_{(n-1,1)}(C(n,k;r))$  is a direct corollary of Theorem 5.2, Lemma 5.3 and Lemma 3.2 in \cite{SZ}. 


 \begin{lem}\label{lem:eigenvalues_standard_2}
 Let  $H=C(n, k ; r)$, where $n \ge 5$ and $2 \le r < k < n$. Then the distinct eigenvalues of $H^{+}$ on the standard representation of $S_n$ are
 $$
 \begin{aligned}
 &  \alpha_1(\rho_{(n-1,1)}(H))=(k-2) !\binom{n-r}{k-r}\frac{1}{n-r}\left((k-1)(n-k)-\frac{(k-r-1)(k-r)}{n-r-1}\right), \\
 & \alpha_2(\rho_{(n-1,1)}(H))=(k-2) !\binom{n-r}{k-r}\left(\frac{r(n-k)}{n-r}-1\right), \text { and } \\ 
 & \alpha_3(\rho_{(n-1,1)}(H))=-(k-2) !\binom{n-r}{k-r}.
 \end{aligned}
 $$	
 \end{lem}

 \begin{lem}\label{lem:eigenvalues_standard_1}
 	Let $H=C(n, k ; 1)$, where $n \ge 5$ and $2 \le k < n$. Then the distinct eigenvalues of $H^{+}$ on the standard representation of $S_n$ are
 $$
 \begin{aligned}
 &  \alpha_1(\rho_{(n-1,1)}(H))=(k-2) !\binom{n-1}{k-1}\frac{1}{n-1}\left((k-1)(n-k)-\frac{(k-2)(k-1)}{n-2}\right), \text{and} \\
 & \alpha_2(\rho_{(n-1,1)}(H))=-(k-2) !\binom{n-2}{k-2}
 .
 \end{aligned}
 $$	
 \end{lem}

The following lemma gives the multiplicities of the eigenvalues of $\rho_{(n-1,1)}(C(n,n-1;r))$.   The proof here  is similar to that of Lemma 6.1 in \cite{SZ}.

\begin{lem}\label{lem:multiplicity_n-1}
	Let $n\ge 5$ and $H=C(n,n-1;r)$. When $2\leq r\leq n-2$, the spectrum of $H^+$ on the standard representation of $S_n$ are
$$\mathrm{Spec}(\rho_{(n-1,1)}(H))=\begin{pmatrix}
	r(n-3)! & (2r-n)(n-3)! & (r-n)(n-3)!\\
	n-r-1 &1&r-1
\end{pmatrix}.$$
When $r=1$, the spectrum of $H^+$ on the standard representation of $S_n$ are
$$\mathrm{Spec}(\rho_{(n-1,1)}(H))=\begin{pmatrix}
	(n-3)! &  -(n-2)!\\
	n-2 &1
\end{pmatrix}.$$
\end{lem}

\begin{proof}
	Suppose $2\leq r\leq n-2$. From \cite[Theorem 5.2]{SZ}, we obtain that all the eigenvalues of $H^+$ on the natural permutation module $M^{(n-1,1)}$ are $|H|$, $r(n-3)!$, $(2r-n)(n-3)!$ and $(r-n)(n-3)!$. Clearly, the largest eigenvalue $|H|$ is simple. Lemma 3.2 in \cite{SZ} indicates that the multiplicities $x,~y,~z$ of the other three eigenvalues are such that $\{x,y,z\}=\{1,~n-r-1,~r-1\}$. For each $\sigma\in S_n$, the trace of $\sigma$ on $M^{(n-1,1)}$ is the number of fixed points of $\sigma$. Thus the trace of $H^+$ on $M^{(n-1,1)}$ is $|H|$. Then we have 
	\begin{equation*}
		|H|+ x\cdot r(n-3)!+y\cdot (2r-n)(n-3)!+ z\cdot (r-n)(n-3)!=|H|.
	\end{equation*}
	If we take $x=n-r-1$, $y=1$ and $z=r-1$, then the quality holds. In addition, the equality fails for any other choice of $x,~y,~z$ such that $\{x,y,z\}=\{1,~n-r-1,~r-1\}$. 

	The proof for the case $r=1$ is similar.
\end{proof}

Finally, we cite the following lemma, which indicates that $\mathrm{Cay}(S_n,C(n,k;r))$ is connected if $k$ is even and has two connected components each isomorphic to $\mathrm{Cay}(A_n,C(n,k;r))$ if $k$ is odd.
\begin{lem}{\upshape \cite[Lemma 5.1]{SZ}} \label{lem:generating_subgroup}
 	Let $1 \leq r<k \leq n$ and let $X$ be the smallest subgroup of $S_n$ containing $C(n, k ; r)$. Then $X=S_n$ if $k$ is even, and $X=A_n$ if $k$ is odd.
 \end{lem}


\section{ $\mathbf{C(n,n-1;r)}$ }
\label{sec:proof_for_theorem_n-1}

For every $i\in [n]$, the symmetric group $\mathrm{Sym_{[n]\setminus \{i\}}}$ is isomorphic to $S_{n-1}$ under the isomorphism $f_i:\mathrm{Sym_{[n]\setminus \{i\}}} \to S_{n-1},~g \mapsto (n~i)g(n~i)$, and $f_i$ sends \[P_i:=C(n,n-1)\cap G_i\] to $C(n-1,n-1)$. Thus for any $\zeta\vdash n$, we have $\rho_\zeta(C(n-1,n-1))=M_i \rho_\zeta(P_i ) M_i^{-1}$ with $M_i=\rho_\zeta((n~i))$. This indicates that $\rho_\zeta(P_1 )$, $\rho_\zeta(P_2 )$, $\ldots$, $\rho_\zeta(P_n )$ share the same spectrum with $\rho_\zeta(C(n-1,n-1))$ for any fixed $\zeta\vdash n$. Now we use Branching Rule to calculate the eigenvalues of $\rho_{\zeta}(C(n-1,n-1))$ for every irreducible representation $\rho_\zeta$ of $S_n$ except for the trivial representation $\rho_{(n)}$ and the standard representation $\rho_{(n-1,1)}$. The results are recorded in the following lemma.

\begin{lem}\label{lem:n-1_cycle_eigenvalues}
Let $H=C(n-1,n-1)$, where $n\ge 7$, and let $\zeta\vdash n$ be such that $\zeta\ne (n)$ or $(1^n)$.  
\begin{enumerate}[{\rm(a)}]
 		\item If $\zeta$ is not a hook or a near hook, then all the eigenvalues of $\rho_{\zeta}(H)$ are $0$. 
 		\item If $\zeta=(n-m,1^m)$ is a hook with $1\leq m\leq n-2$, then the distinct eigenvalues of $\rho_\zeta(H)$ are $(-1)^mm!(n-2-m)!$ and $(-1)^{m-1}(m-1)!(n-1-m)!$.
        \item If $\zeta=(n-m,2,1^{m-2})$ with $2\leq m\leq n-2$, then the distinct eigenvalues of $\rho_\zeta(H)$ are $(-1)^{m-1}(m-1)!(n-1-m)!$ and $0$.
 	\end{enumerate} 	
\end{lem}

\begin{proof}
	Since $H=C(n-1,n-1)$ is normal in $S_{n-1}$, we obtain by Branching Rule  $\rho_{\zeta}(H)=\bigoplus_{\zeta^-} \rho_{\zeta^-}(H) $, and by Schur's Lemma each $\rho_{\zeta^-}(H)$ is a scalar matrix. We then derive from the definition of normalized characters and the fact that they are class functions that $\rho_{\zeta^-}(H)=|H|\cdot \tilde{\chi}_{\zeta^-}(\sigma)\cdot \mathcal{I}_m$, where $\sigma$ is any $(n-1)$-cycle of $S_{n-1}$ and $\mathcal{I}_m$ is the identity matrix with dimension $m=\mathrm{dim}\rho_{\zeta^-}$. Thus the distinct eigenvalues of $\rho_{\zeta}(H)$ are $\{ |H|\cdot \tilde{\chi}_{\zeta^-}(\sigma)  \}_{\zeta^-}$, where $\tilde{\chi}_{\zeta^-}(\sigma)$ can be calculated by Lemma~\ref{lem:n-cycle}.
    \begin{enumerate}[{\rm(a)}]
    	\item When $\zeta$ is not a hook or a near hook, $\zeta^-$ is never a hook. From Lemma~\ref{lem:n-cycle}, the normalized character of $\zeta^-$ on $\sigma$ is $0$. Then the eigenvalues of $\rho_{\zeta}(H)$ are all $0$'s.
    	\item If $\zeta=(n-m,1^m)$ is a hook with $1\leq m\leq n-2$, then $\zeta^-=(n-1-m,1^m)$ or $(n-m,1^{m-1})$. By Lemma~\ref{lem:n-cycle}, we have $|H|\cdot \tilde{\chi}_{(n-1-m,1^m)}(\sigma)=(-1)^m m!(n-2-m)! $ and $|H|\cdot \tilde{\chi}_{(n-m,1^{m-1})}(\sigma)=(-1)^{m-1}(m-1)!(n-1-m)! $. Thus the distinct eigenvalues of $\rho_{\zeta}(H)$ are $(-1)^m m!(n-2-m)!$ and $(-1)^{m-1}(m-1)!(n-1-m)!$.
    	\item If $\zeta=(n-m,2,1^{m-2})$ with $2\leq m\leq n-2$, then $\zeta^-=(n-m,1^{m-1})$, $(n-m-1,2,1^{m-2})$ if $m\leq n-3$, or $(n-m,2,1^{m-3})$ if $m\ge 3$. The last two values of $\zeta^-$ attain $0$ characters on $\sigma$ by Lemma~\ref{lem:n-cycle} and $|H|\cdot \tilde{\chi}_{(n-m,1^{m-1})}(\sigma)=(-1)^{m-1}(m-1)!(n-1-m)! $. Thus the distinct eigenvalues of $\rho_{\zeta}(H)$ are $(-1)^{m-1}(m-1)!(n-1-m)!$ and $0$.
    \end{enumerate}
    This completes the proof.
\end{proof}

\begin{lem}\label{lem:C(n,n-1;r)}
Let $H=C(n,n-1;r)$, where $n\ge 7$ and $r\in \{1,2,\ldots n-2\}$, and let $\zeta\vdash n$ be such that $\zeta\ne (n)$, $(n-1,1)$, $(2,1^{n-2})$ or $(1^n)$. 
	\begin{enumerate}[{\rm(a)}]
      \item If $n$ is odd and $1\leq r\leq 2$, then $ \lambda_1(\rho_\zeta(H))\leq 2(n-2)(n-4)! $.

      \item If $n$ is odd and  $3\leq r< n/2$, then $
      \lambda_1(\rho_\zeta(H))\leq r(n-3)!$.
      \item If $n$ is odd and $n/2<r\leq n-2$, then $\lambda_1(\rho_\zeta(H))\leq  (n-r)(n-3)!$.
      \item If $n$ is even and $1\leq r\leq 2$, then $\lambda_1(\rho_\zeta(H))\leq r(n-3)!$; moreover, if in addition $\zeta\ne (n-2,1^2)$ or $(3,1^{n-3})$, then $\lambda_1(\rho_\zeta(H))< r(n-3)!$.
      \item If $n$ is even and $3\leq r\leq n-2$, then $\lambda_1(\rho_\zeta(H))\leq 2(n-r)(n-4)!$.
    \end{enumerate}
\end{lem}

\begin{proof}
    Recall $P_i=C(n,n-1)\cap G_{i}$ for every $i\in [n]$. Then $P_i$ is closed under inverse and conjugation in $\mathrm{Sym}_{[n]\setminus\{i\}}$ and these $P_i$'s are mutually disjoint. For any $\zeta\vdash n$ and any $i\in [n]$, the matrix $\rho_\zeta(P_i)$ is symmetric, guaranteed by Lemma~\ref{lem:symmetric}, and shares the same eigenvalues with $\rho_\zeta(C(n-1,n-1))$ by the analysis above Lemma~\ref{lem:n-1_cycle_eigenvalues}.

	We decompose $H$ into $n-r$ parts as $H=\dot\cup_{i=r+1}^n P_i$. Then for any $\zeta\vdash n$ we have $\rho_\zeta(H)=\sum_{i=r+1}^n \rho_\zeta(P_i)$  and by Weyl Inequalities, the following bound for $\lambda_1(\rho_\zeta(H))$ holds:
	\begin{gather}
		\lambda_1(\rho_\zeta(H))\leq (n-r)\lambda_1(\rho_\zeta(C(n-1,n-1))). \label{eq:max_weyl_Inequalities_for_n-1} 
	\end{gather}
      Now substitute the results of Lemma~\ref{lem:n-1_cycle_eigenvalues} into \eqref{eq:max_weyl_Inequalities_for_n-1}.
 	\begin{itemize}
		\item[$\bullet$]  If $\zeta$ is not a hook or a near hook, then $\lambda_1(\rho_\zeta(H))\leq 0$.
		\item[$\bullet$] If $\zeta=(n-m,1^m)$ with $2\leq m\leq n-3$ and $m$ is even, then \[\lambda_1(\rho_\zeta(H))\leq (n-r)m!(n-2-m)!\leq  \begin{cases}
			(n-r)(n-3)!, \text{ if } n \text{ is odd};\\
			2(n-r)(n-4)!,\text{ if } n \text{ is even}.
		\end{cases} \]

        \item[$\bullet$] If $\zeta=(n-m,1^m)$ with $2\leq m\leq n-3$ and $m$ is odd, then $\lambda_1(\rho_\zeta(H))\leq (n-r)(m-1)!(n-1-m)!\leq 2(n-r)(n-4)!$. 

        \item[$\bullet$] If $\zeta=(n-m,2,1^{m-2})$ with $2\leq m\leq n-2$ and $m$ is even, then $\lambda_1(\rho_\zeta(H))\leq 0$. 

        \item[$\bullet$] If $\zeta=(n-m,2,1^{m-2})$ with $2\leq m\leq n-2$ and $m$ is odd, then \[\lambda_1(\rho_\zeta(H))\leq (n-r)(m-1)!(n-1-m)!\leq \begin{cases}
        	(n-r)(n-3)!,\text{ if }n \text{ is odd};\\
        	2(n-r)(n-4)!,\text{ if }n\text{ is even}.
        \end{cases} \] 



	\end{itemize}
	To sum up, when $n$ is odd, the largest eigenvalue of $\rho_\zeta(H)$ is no larger than $(n-r)(n-3)!$; when $n$ is even, the largest eigenvalue of $\rho_\zeta(H)$ is no larger than $2(n-r)(n-4)!$. This completes the proof of (c) and (e).

    We now express $H$ as $H=C(n,n-1)\setminus (\dot\cup_{i=1}^r P_i)$. Then for any $\zeta\vdash n$, we have $\rho_\zeta(H)=\rho_\zeta(C(n,n-1))-\sum_{i=1}^i \rho_\zeta(P_i)$.  By Weyl Inequalities, we have the following bound for $\lambda_1(\rho_\zeta(H))$:
    \begin{gather}
		\lambda_1(\rho_\zeta(H))\leq \lambda_1(\rho_\zeta(C(n,n-1)))-r \cdot \lambda_{min}(\rho_\zeta(C(n-1,n-1))) . \label{eq:max_weyl_Inequalities_r=1} 
	\end{gather}
    As $C(n,n-1)$ is a conjugacy class of $S_n$, the matrix $\rho_\zeta(C(n,n-1))$ is a scalar matrix and the unique eigenvalue of $\rho_\zeta(C(n,n-1))$ is given by $|C(n,n-1)|\cdot \tilde{\chi}_{\zeta}(\sigma)$, where $\sigma$ is any element in $C(n,n-1)$ and $\tilde{\chi}_{\zeta}(\sigma)$ can be calculated by Lemma~\ref{lem:n-1_cycles}.  Substituting Lemma~\ref{lem:n-1_cycle_eigenvalues} into \eqref{eq:max_weyl_Inequalities_r=1} gives us the following results: 
	\begin{itemize}
		\item[$\bullet$] Suppose $\zeta$ is not a hook or a near hook. The eigenvalues of $\rho_\zeta(C(n,n-1))$ and $\rho_\zeta(C(n-1,n-1))$ are all $0$. Thus from \eqref{eq:max_weyl_Inequalities_r=1} we derive that $\lambda_1(\rho_\zeta(H))\leq 0$.

		\item[$\bullet$] If $\zeta=(n-m,1^m)$ with $2\leq m\leq n-3$ and $m$ is even, then all eigenvalues of $\rho_\zeta(C(n,n-1))$ are $0$ and thus $\lambda_1(\rho_\zeta(H))\leq r\cdot (m-1)!(n-1-m)! \leq r(n-3)! $. If in addition $\zeta\ne(n-2,1^2)$, then $\lambda_1(\rho_\zeta(H))<r(n-3)!$.

		

		\item[$\bullet$] If $\zeta=(n-m,1^m)$ with $2\leq m\leq n-3$ and $m$ is odd, then we have \[\lambda_1(\rho_\zeta(H))\leq r\cdot m!(n-2-m)!\leq \begin{cases}
		2r(n-4)!,&\text{ if }n\text{ is odd};\\
		r(n-3)!,&\text{ if }n\text{ is even}.	
		\end{cases}  \] Here if in addition $n$ is even and $\zeta\ne(3,1^{n-3})$, then $\lambda_1(\rho_\zeta(H))<r(n-3)!$.


        \item[$\bullet$] If $\zeta=(n-m,2,1^{m-2})$ with $2\leq m\leq n-2$ and $m$ is even, then the unique eigenvalue of $\rho_\zeta(C(n,n-1))$ is $-(n-m)(n-2-m)!m(m-2)!$ and thus $\lambda_1(\rho_\zeta(H))\leq -(n-m)(n-2-m)!m(m-2)!  +r\cdot (m-1)!(n-1-m)! < (r-1)\cdot (m-1)!(n-1-m)!<(r-1)(n-3)!$. 

        \item[$\bullet$] If $\zeta=(n-m,2,1^{m-2})$ with $2\leq m\leq n-2$ and $m$ is odd, then the unique eigenvalue of $\rho_\zeta(C(n,n-1))$ is $(n-m)(n-2-m)!m(m-2)!$ and thus \[\lambda_1(\rho_\zeta(H))\leq (n-m)(n-2-m)!m(m-2)\leq \begin{cases}
        	 2(n-2)(n-4)!,&\text{ if }n\text{ is odd};\\
        	 3(n-3)(n-5)!,&\text{ if }n\text{ is even}.
        \end{cases}  \]

         


	\end{itemize}
	If $n$ is odd and $1\leq r\leq 2$, then $\lambda_1(\rho_\zeta)\leq \min\{(n-r)(n-3)!,2(n-2)(n-4)!\}=2(n-2)(n-4)!$.
	If $n$ is odd and $3\leq r<n/2$, then $\lambda_1(\rho_\zeta)\leq \min\{(n-r)(n-3)!,r(n-3)!\}=r(n-3)!$.  If $n$ is even and $1\leq r\leq 2$, then $\lambda_1(\rho_\zeta)\leq \min\{2(n-r)(n-4)!,r(n-3)!\}=r(n-3)!$; moreover, if in addition $\zeta\ne (n-2,1^2)$ or $(3,1^{n-3})$, then $\lambda_1(\rho_\zeta)<r(n-3)!$.
\end{proof}

\begin{thm}\label{thm:n-1_odd}
Let $\Gamma=\mathrm{Cay}(S_n,C(n,n-1;r))$ with $n\ge 5$ odd and $r\in \{1,2,\ldots n-2\}$. Then $\Gamma$ is connected and bipartite with $\lambda_2(\Gamma)=\alpha_2(\Gamma)$ such that the following statements hold:
\begin{enumerate}[{\rm (a)}]
    \item If $n=5$ and $r=1$,  then $\lambda_2(\Gamma)=6$ with multiplicity $14$, attained exactly by $(2,1^3)$ and $(2,2,1)$.
    \item If $n\ge 7$ and $r=1$, then $\lambda_2(\Gamma)=(n-2)!$ with multiplicity $n-1$, attained uniquely by $(2,1^{n-2})$.
    \item If $2\leq r<n/2$, then $\lambda_2(\Gamma)=(n-r)(n-3)!$ with multiplicity $(n-1)(r-1)$, attained uniquely by $(2,1^{n-2})$.
	\item If $n/2<r<n-1$, then $\lambda_2(\Gamma)=r(n-3)!$ with multiplicity $(n-1)(n-r-1)$, attained uniquely by $(n-1,1)$.
\end{enumerate}

\end{thm}
\begin{proof}
	Let $H=C(n,n-1;r)$ with $1\leq r\leq n-2$. Since $n$ is odd, all permutations in $H$ are odd and thus $\Gamma=\mathrm{Cay}(S_n,H)$ is a connected regular bipartite graph with $\lambda_2(\Gamma)=\alpha_2(\Gamma)$. When $n=5$ and $1\leq r\leq n-2$, we verify $\lambda_2(\Gamma)$ by computation in \textsc{Magma} \cite{BCP}. For the remainder of this proof, suppose $n\ge 7$. 

	It is clear that if $\zeta=(n)$, then matrix $\rho_\zeta(H)$ has only one eigenvalue, which is $|H|=(n-r)(n-2)!$; if $\zeta=(1^n)$, then $\rho_\zeta$ is the sign representation and sends $H$ to $-|H|$. With $\Gamma$ a connected regular bipartite graph, we know that $|H|$ and $-|H|$ are the simple largest and smallest eigenvalues of $\Gamma$, respectively. 

    By Lemma~\ref{lem:multiplicity_n-1}, we get $\lambda_1(\rho_{(n-1,1)}(H))=r(n-3)!$ with multiplicity $a:=n-r-1$   and \[\lambda_{\min}(\rho_{(n-1,1)}(H))=\begin{cases}
    -(n-r)(n-3)!,&\text{ if }2\leq r\leq n-2;\\
    -(n-2)!,&\text{ if }r=1	
    \end{cases}  \] with multiplicity \[b:=\begin{cases}
    	r-1,&\text{ if }2\leq r\leq n-2;\\
    	1,&\text{ if }r=1.
    \end{cases}\] Since Lemma~\ref{lem:transpose} implies that $\rho_{(2,1^{n-2})}$ only differs from $\rho_{(n-1,1)}$ at odd permutations by the sign, we have  $\rho_{(2,1^{n-2})}(H)=-\rho_{(n-1,1)}(H)$, which implies \[\lambda_1(\rho_{(2,1^{n-2})}(H))=\begin{cases}
    (n-r)(n-3)!,&\text{ if }2\leq r\leq n-2;\\
    (n-2)!,&\text{ if }r=1	
    \end{cases} \] with multiplicity $b$ and $\lambda_{\min}(\rho_{(2,1^{n-2})}(H))=-r(n-3)!$ with multiplicity $a$. 

    First suppose $r=1$ or $2$. From Lemma~\ref{lem:C(n,n-1;r)} (a) we know that if $\zeta\vdash n$ and $\zeta\ne (n),(1^n),(n-1,1)$ or $(2,1^{n-2})$, then $\lambda_{1}(\rho_\zeta(H))\leq 2(n-2)(n-4)!$, which is strictly smaller than $(n-2)!$. Thus $\lambda_2(\Gamma)=(n-2)!$, attained uniquely by $(2,1^{n-2})$. 

    Next suppose $3\leq r<n/2$. We have verified in Lemma~\ref{lem:C(n,n-1;r)} (b) that if $\zeta\vdash n$ and $\zeta\ne (n)$, $(n-1,1)$, $(2,1^{n-2}) $ or $(1^n)$, then $\lambda_1(\rho_\zeta(H))\leq r(n-3)!$, which is strictly smaller than $(n-r)(n-3)!$. Thus $\lambda_2(\Gamma)=(n-r)(n-3)!$, attained uniquely by $(2,1^{n-2})$.

     Now suppose $n/2< r\leq n-2$. We have verified in Lemma~\ref{lem:C(n,n-1;r)} (c) that if $\zeta\vdash n$ and $\zeta\ne (n)$, $(n-1,1)$, $(2,1^{n-2}) $ or $(1^n)$, then $\lambda_1(\rho_\zeta(H))\leq (n-r)(n-3)!$, which is strictly smaller than $r(n-3)!$. Thus $\lambda_2(\Gamma)=r(n-3)!$, attained uniquely by $(n-1,1)$. 

	 Finally, we apply Proposition~\ref{prop1.1} to derive the conclusion on the multiplicity of $\lambda_2(\Gamma)$. If $r=1$, then $\lambda_2(\Gamma)=(n-2)!$ is attained uniquely by $(2,1^{n-2})$ and has  multiplicity $d_{(2,1^{n-2})}\cdot b=n-1$.  If $2\leq r<n/2$, then $\lambda_2(\Gamma)=(n-r)(n-3)!$ is attained uniquely by $(2,1^{n-2})$ and has multiplicity $d_{(2,1^{n-2})}\cdot b=(n-1)(r-1)$.  If $n/2< r\leq n-2$, then $\lambda_2(\Gamma)=r(n-3)!$ is attained uniquely by $(n-1,1)$ and has multiplicity  $d_{(n-1,1)}\cdot a=(n-1)(n-r-1)$. 
\end{proof}

The following corollary is about the two smallest eigenvalues of $\mathrm{Cay}(S_n,C(n,n-1;r))$ with odd $n$. 
\begin{cor}\label{cor:odd_small}
  Let $\Gamma=\mathrm{Cay}(S_n,C(n,n-1;r))$ with $n\ge 5$ odd and $r\in \{1,2,\ldots,n-2\}$. The smallest eigenvalue of $\Gamma$ is $-(n-r)(n-2)!$, which is simple and attained  uniquely by $(1^n)$.
\begin{enumerate}[{\rm (a)}]
	\item If $n=5$ and $r=1$,  then the second smallest eigenvalue of $\Gamma$ is $-6$ with multiplicity $14$, attained exactly by $(4,1)$ and $(3,2)$.
    \item If $n\ge 7$ and $r=1$, then the second smallest eigenvalue of $\Gamma$ is $-(n-2)!$ with multiplicity $n-1$, attained uniquely by $(n-1,1)$.
    \item If $2\leq r<n/2$, then  the second smallest eigenvalue of $\Gamma$ is  $-(n-r)(n-3)!$ with multiplicity $(n-1)(r-1)$, attained uniquely by $(n-1,1)$.
	\item If $n/2<r<n-1$, then the second smallest eigenvalue of $\Gamma$ is $-r(n-3)!$ with multiplicity $(n-1)(n-r-1)$, attained uniquely by $(2,1^{n-2})$.
\end{enumerate}
\end{cor}
\begin{proof}
Note that $\Gamma$ is a connected bipartite graph when $n$ is odd. Thus the spectrum  of $\Gamma$ is symmetric about $0$, that is, if $\lambda$ is an eigenvalue of $\Gamma$ with multiplicity $m_\lambda$, then $-\lambda$ is also an eigenvalue of $\Gamma$ with multiplicity $m_\lambda$. Since $n$ is odd, all the permutations in $C(n,n-1;r)$ are odd and thus by Lemma~\ref{lem:transpose}, $\rho_{\zeta'}(H)=-\rho_\zeta(H)$ for any $\zeta\vdash n$. Hence $\lambda$ is an eigenvalue of $\rho_\zeta(H)$ with multiplicity $m_\zeta^\lambda$ if and only if $-\lambda$ is an eigenvalue of $\rho_{\zeta'}(H)$ with multiplicity $m_\zeta^\lambda$. Note also that $\rho_\zeta$ and $\rho_{\zeta'}$ have the same dimension by Lemma~\ref{lem:Hook}. Now this corollary follows from Theorem~\ref{thm:n-1_odd} and Proposition~\ref{prop1.1}.
\end{proof}

\begin{thm}\label{thm:n-1_even}
  Let $\Gamma=\mathrm{Cay}(S_n,C(n,n-1;r))$ with $n\ge 6$ even and $r\in\{1,2,\ldots,n-2\}$. 
  \begin{enumerate}[{\rm (a)}]
  	\item If $n=6$ and $r=1$, then $\alpha_2(\Gamma)=9$ with multiplicity $160$, attained uniquely by $(3,2,1)$.  
  	\item If $(n,r)\ne(6,1)$, then $\alpha_2(\Gamma)=r(n-3)!$. Moreover, the following statements hold: 
          \begin{enumerate}[{\rm (b.1)}]
          	\item If $r\leq 2$, then $\alpha_2(\Gamma)$ is attained by $(n-1,1)$ and $(2,1^{n-2})$, and can only be attained by $(n-1,1)$, $(2,1^{n-2})$, $(n-2,1^{2})$ and $(3,1^{n-3})$.
          	\item If $3\leq r\leq n-2$, then $\alpha_2(\Gamma)$ has multiplicity $2(n-1)(n-r-1)$ and is attained exactly by $(n-1,1)$ and $(2,1^{n-2})$.
          \end{enumerate}
  \end{enumerate}
\end{thm}

\begin{proof}Since $n$ is even, all permutations in $H=C(n,n-1;r)$ are even. The graph $\Gamma=\mathrm{Cay}(S_n,H)$ has two connected components and both are isomorphic to $\mathrm{Cay}(A_n,H)$. Thus the degree $|H|=(n-r)(n-2)!$ is the largest eigenvalue of $\Gamma$ with multiplicity $2$, attained by $\rho_{(n)}$ and $\rho_{(1^n)}$ simultaneously. When $n=6$ and $1\leq r\leq 4$, we verify $\alpha_2(\Gamma)$ via computation in \textsc{Magma} \cite{BCP}. Now suppose $n\ge 8$. 

    We write $a=n-r-1$. By Lemma~\ref{lem:multiplicity_n-1}, we get $\lambda_1(\rho_{(n-1,1)}(H))=r(n-3)!$ with multiplicity $a$.  Since Lemma~\ref{lem:transpose} implies that  $\rho_{(2,1^{n-2})}(H)=\rho_{(n-1,1)}(H)$, we have $\lambda_1(\rho_{(2,1^{n-2})}(H))=\lambda_1(\rho_{(n-1,1)}(H))=r(n-3)!$ with multiplicity $a$.

	First suppose $1\leq r\leq 2$. According to Lemma~\ref{lem:C(n,n-1;r)} (d), if $\zeta\vdash n$ is such that $\zeta\ne(n)$, $(n-1,1)$, $(2,1^{n-2})$ or $(1^n)$, then the eigenvalue $\lambda_1(\rho_\zeta(H))\leq r(n-3)!$, and if in addition $\zeta\ne (n-2,1^2)$ or $(3,1^{n-3})$, then $\lambda_1(\rho_\zeta(H))< r(n-3)!$. Thus we conculde that  $\alpha_2(\Gamma)=r(n-3)!$, which is attained by $(n-1,1)$ and $(2,1^{n-2})$, and can only be attained by $(n-1,1)$, $(2,1^{n-2})$, $(n-2,1^2)$ and $(3,1^{n-3})$.

	Next suppose $3\leq r\leq n-2$. Lemma~\ref{lem:C(n,n-1;r)} (e) implies that, when $\zeta\vdash n$ such that $\zeta\ne(n)$, $(n-1,1)$, $(2,1^{n-2})$ or $(1^n)$, the eigenvalue $\lambda_1(\rho_\zeta(H))\leq 2(n-r)(n-4)!<r(n-3)!$. Thus we conclude that $\alpha_2(\Gamma)=r(n-3)!$, which is attained exactly by $(n-1,1)$ and  $(2,1^{n-2})$, and that its multiplicity equals $d_{(n-1,1)}\cdot a+d_{(2,1^{n-2})}\cdot a=2(n-1)(n-r-1)$ by Proposition~\ref{prop1.1}.
    \end{proof}

\begin{re}
	In item (b.1) of the above theorem, we cannot rule out $(n-2,1^2)$ and $(3,1^{n-3})$ for attaining $\alpha_2(\mathrm{Cay}(S_n,C(n,n-1;r)))$. For example,  the strictly second largest eigenvalue of $\mathrm{Cay}(S_6,C(6,5;2))$ is $12$ with multiplicity $50$, attained exactly by 
	$(5,1)$, $(4,1^2)$,$(3,1^3)$ and $(2,1^{4})$;
	 for $\mathrm{Cay}(S_8,C(8,7;r))$ with $r=1$ or $2$, its strictly second largest eigenvalue is attained exactly by 
     $(7,1)$, $(6,1^2)$, $(3,1^5)$ and $(2,1^6)$.
	 However, we also cannot confirm that $(n-2,1^2)$ and $(3,1^{n-3})$ always attain $\alpha_2(\mathrm{Cay}(S_n,C(n,n-1;r)))$ in the case of (b.1).
\end{re}

\begin{lem}\label{lem:n-1_even_smallest}
Let $H=C(n,n-1;r)$, where $n\ge 7$ and $r\in \{1,2,\ldots n-2\}$, and let $\zeta\vdash n$ be such that $\zeta\ne (n)$, $(n-1,1)$, $(2,1^{n-2})$ or $(1^n)$. The following bounds for $\lambda_{\min}(\rho_\zeta(H))$ hold:\begin{enumerate}[{\rm (a)}]
\item If $n$ is odd and $1\leq r\leq 2$, then $\lambda_{\min}(\rho_\zeta(H))\ge -2(n-2)(n-4)!$.
\item If $n$ is odd and $3\leq r< n/2$, then $\lambda_{\min}(\rho_\zeta(H))\ge -r(n-3)!$.
\item If $n$ is odd and $n/2\leq r\leq n-2$, then $\lambda_{\min}(\rho_\zeta(H))\ge -(n-r)(n-3)!$.

   \item If $n$ is even and $1\leq r\leq n-3$, then $\lambda_{\min}(\rho_\zeta(H))\ge  -2(n-2)(n-4)!$.
   \item If $n$ is even and $r=n-2$, then $\lambda_{\min}(\rho_\zeta(H))\ge -(n-r)(n-3)!$; moreover, if in addition $\zeta\ne (n-2,1^2)$, $(3,1^{n-3})$, $(n-2,2)$ or $(2,2,1^{n-4})$, then $\lambda_{\min}(\rho_\zeta(H))> -(n-r)(n-3)!$.
\end{enumerate}
\end{lem}

\begin{proof} Recall $P_i=C(n,n-1)\cap G_{i}$ for every $i\in [n]$, and that for any $\zeta\vdash n$ and any $i\in [n]$, the matrix $\rho_\zeta(P_i)$ is symmetric and has the same eigenvalues as $\rho_\zeta(C(n-1,n-1))$.

We apply Weyl Inequalities to two different decompositions of $H$. The first decomposition is $H=\dot\cup_{i=r+1}^n P_i$. Applying Weyl Inequalities, we obtain: 
\begin{gather}
		\lambda_{\min}(\rho_\zeta(H))\ge (n-r)\lambda_{\min}(\rho_\zeta(C(n-1,n-1))) \label{eq:min_weyl_Inequalities_for_n-1}.
	\end{gather}
Now we substitute Lemma~\ref{lem:n-1_cycle_eigenvalues} to \eqref{eq:min_weyl_Inequalities_for_n-1}.
\begin{itemize}
	\item[$\bullet$]   If $\zeta$ is not a hook or a near hook, then $\lambda_{\min}(\rho_\zeta(H))\ge 0$.

		\item[$\bullet$] If $\zeta=(n-m,1^m)$ with $2\leq m\leq n-3$ and $m$ is even, then $\lambda_{\min}(\rho_\zeta(H))\ge -(n-r)(m-1)!(n-1-m)!\ge -(n-r)(n-3)!$, and if in addition $\zeta\ne (n-2,1^2)$, then $\lambda_{\min}(\rho_\zeta(H))> -(n-r)(n-3)!$. 

        \item[$\bullet$] If $\zeta=(n-m,1^m)$ with $2\leq m\leq n-3$ and $m$ is odd, then \[\lambda_{\min}(\rho_\zeta(H))\ge -(n-r)m!(n-2-m)!\ge \begin{cases}
              -2(n-r)(n-4)!,&\text{ if }n\text{ is odd};\\
           	-(n-r)(n-3)!,&\text{ if }n\text{ is even}.
        \end{cases} \]
         When $n$ is even, if in addition $\zeta\ne (3,1^{n-3} )$, then $\lambda_{\min}(\rho_\zeta(H))>-(n-r)(n-3)!$.

        \item[$\bullet$] If $\zeta=(n-m,2,1^{m-2})$ with $2\leq m\leq n-2$ and $m$ is even, then $\lambda_{\min}(\rho_\zeta(H))\ge -(n-r)(m-1)!(n-1-m)!\ge-(n-r)(n-3)!$. When $n$ is even, if in addition $\zeta\ne(n-2,2)$ or $(2,2,1^{n-4})$, then $\lambda_{\min}(\rho_\zeta(H))>-(n-r)(n-3)!$.


        \item[$\bullet$] If $\zeta=(n-m,2,1^{m-2})$ with $2\leq m\leq n-2$ and $m$ is odd, then $\lambda_{\min}(\rho_\zeta(H))\ge 0$.
 \end{itemize}
 To sum up, we have $\lambda_{\min}(\rho_\zeta(H))\ge -(n-r)(n-3)!$. For even $n$, if in addition $\zeta\ne (n-2,1^2)$,  $(3,1^{n-3})$, $(n-2,2)$ or $(2,2,1^{n-4})$, then $\lambda_{\min}(\rho_\zeta(H))> -(n-r)(n-3)!$.

The second decomposition of $H$ is $H=C(n,n-1)\setminus (\dot\cup_{i=1}^r P_i)$. This implies $\rho_\zeta(H)=\rho_\zeta(C(n,n-1))-\sum_{i=1}^i \rho_\zeta(P_i)$ for any $\zeta\vdash n$.  By Weyl Inequalities, we have:
    \begin{gather}
		\lambda_{\min}(\rho_\zeta(H))\ge \lambda_{\min}(\rho_\zeta(C(n,n-1)))-r\cdot \lambda_1(\rho_\zeta(C(n-1,n-1))) \label{eq:min_weyl_Inequalities_r=1}.
	\end{gather}
    The unique eigenvalue of $\rho_\zeta(C(n,n-1))$ is given by $|C(n,n-1)|\cdot \tilde{\chi}_{\zeta}(\sigma)$, where $\sigma$ is any element in $C(n,n-1)$ and $\tilde{\chi}_{\zeta}(\sigma)$ can be calculated by Lemma~\ref{lem:n-1_cycles}.  Substituting Lemma~\ref{lem:n-1_cycle_eigenvalues} into \eqref{eq:min_weyl_Inequalities_r=1} gives us the following results:
\begin{itemize}
		\item[$\bullet$] If $\zeta$ is not a hook or a near hook, then the eigenvalues of $\rho_\zeta(C(n,n-1))$ and $\rho_\zeta(C(n-1,n-1))$ are all $0$, and thus by \eqref{eq:min_weyl_Inequalities_r=1} we derive $\lambda_{\min}(\rho_\zeta)\ge 0$.

		\item[$\bullet$] If $\zeta=(n-m,1^m)$ with $2\leq m\leq n-3$ and $m$ is even, then all eigenvalues of $\rho_\zeta(C(n,n-1))$ are $0$ and thus \[\lambda_{\min}(\rho_\zeta(H))\ge -r\cdot m!(n-2-m)!\ge \begin{cases}
           -r(n-3)!, &\text{ if }n\text{ is odd};\\
			-2r(n-4)!,&\text{ if }n\text{ is even}.
		\end{cases}\]



		\item[$\bullet$] If $\zeta=(n-m,1^m)$ with $2\leq m\leq n-3$ and $m$ is odd, then $\lambda_{\min}(\rho_\zeta(H))\ge -r\cdot (m-1)!(n-1-m)!\ge-2r(n-4)!$. 


        \item[$\bullet$] If $\zeta=(n-m,2,1^{m-2})$ with $2\leq m\leq n-2$ and $m$ is even, then the unique eigenvalue of $\rho_\zeta(C(n,n-1))$ is $-(n-m)(n-2-m)!m(m-2)!$ and $\lambda_{\min}(\rho_\zeta(H))\ge -(n-m)(n-2-m)!m(m-2)!\ge-2(n-2)(n-4)!.$ 




        \item[$\bullet$] If $\zeta=(n-m,2,1^{m-2})$ with $2\leq m\leq n-2$ and $m$ is odd, then the unique eigenvalue of $\rho_\zeta(C(n,n-1))$ is $(n-m)(n-2-m)!m(m-2)!$ and 
        \begin{align*}
       \lambda_{\min}(\rho_\zeta(H)) &\ge (n-m)(n-2-m)!m(m-2)!-r\cdot (m-1)!(n-1-m)!\\
       &>-(r-1)\cdot (m-1)!(n-1-m)! \\
        &\ge \begin{cases}
            -(r-1)(n-3)!,&\text{ if }n\text{ is odd};\\
        	-2(r-1)(n-4)!,&\text{ if }n\text{ is even}.
        \end{cases}	
        \end{align*}  
	\end{itemize} To sum up, if $n$ is odd and $1\leq r\leq 2$, then $\lambda_{\min}(\rho_\zeta(H))\ge -2(n-2)(n-4)!$; if $n$ is odd and $3\leq r\leq n-2$, then $\lambda_{\min}(\rho_\zeta(H))\ge -r(n-3)!$; if $n$ is even, then $\lambda_{\min}(\rho_\zeta(H))\ge -2(n-2)(n-4)!$. 

	

	Combining the bounds we built via the two decompositions, we have the following results.  
    If $n$ is odd and $1\leq r\leq 2$, then $\lambda_{\min}(\rho_\zeta(H))\ge \max\{-(n-r)(n-3)!,-2(n-2)(n-4)!\}= -2(n-2)(n-4)!$. If $n$ is odd and $3\leq r< n/2$, then $\lambda_{\min}(\rho_\zeta(H))\ge \max\{-(n-r)(n-3)!,-r(n-3)!\}= -r(n-3)!$. If $n$ is odd and $n/2< r\leq n-2$, then $\lambda_{\min}(\rho_\zeta(H))\ge \max\{-(n-r)(n-3)!,-r(n-3)!\}= -(n-r)(n-3)!$.
    If $n$ is even and $1\leq r\leq n-3$, then $\lambda_{\min}(\rho_\zeta(H))\ge \max\{-(n-r)(n-3)!,-2(n-2)(n-4)!\}= -2(n-2)(n-4)!$. 
    If $n$ is even and $r=n-2$, then $\lambda_{\min}(\rho_\zeta(H))\ge \max\{-(n-r)(n-3)!,-2(n-2)(n-4)!\}=-(n-r)(n-3)!$, and if in addition  $\zeta\ne (n-2,1^2)$, $(3,1^{n-3})$, $(n-2,2)$ or $(2,2,1^{n-4})$, then $\lambda_{\min}(\rho_\zeta(H))> -(n-r)(n-3)!$.
    \end{proof}

\begin{cor}\label{thm:even_small}
   Let $\Gamma=\mathrm{Cay}(S_n,C(n,n-1;r))$ with $n\ge 6$ even and $r\in \{1,2,\ldots,n-2\}$.
  \begin{enumerate}[{\rm (a)}]
    \item If $r=1$, then the smallest eigenvalue of $\Gamma$ is $-(n-2)!$ with multiplicity $2(n-1)$, attained exactly by $(n-1,1)$ and $(2,1^{n-2})$.
    \item If $2\leq r\leq n-3$, then the smallest eigenvalue of $\Gamma$ is $-(n-r)(n-3)!$ with multiplicity $2(n-1)(r-1)$, attained exactly by $(n-1,1)$ and $(2,1^{n-2})$.
  	\item If $r=n-2$, then the smallest eigenvalue of $\Gamma$ is $-(n-r)(n-3)!$, which is  attained by $(n-1,1)$ and $(2,1^{n-2})$ and can only be attained possibly by $(n-1,1)$, $(2,1^{n-2})$, $(n-2,1^2)$, $(3,1^{n-3})$, $(n-2,2)$ and $(2,2,1^{n-4})$.
  \end{enumerate}
\end{cor}

\begin{proof}Since $n$ is even, all permutations in $H=C(n,n-1;r)$ are even. The largest eigenvalue of $\Gamma$ is $|H|=(n-r)(n-2)!$ with multiplicity $2$, attained by $(n)$ and $(1^n)$ simultaneously. When $n=6$ and $1\leq r\leq 4$, we verify the samllest eigenvalue of $\Gamma$ via computation in \textsc{Magma} \cite{BCP}. Now suppose $n\ge 8$.

	By Lemma~\ref{lem:multiplicity_n-1} and Lemma~\ref{lem:transpose}, we have \[\lambda_{\min}(\rho_{(n-1,1)}(H))=\lambda_{\min}(\rho_{(2,1^{n-2})}(H))=\begin{cases}
		-(n-r)(n-3)!, &\text{ if }2\leq r\leq n-2;\\
		-(n-2)!,&\text{ if }r=1
	\end{cases} \]    with multiplicity \[b:=\begin{cases}
		r-1, &\text{ if } 2\leq r\leq n-2;\\
		1,&\text{ if }r=1.
	\end{cases}\]



	First suppose $r=1$. We find from Lemma~\ref{lem:n-1_even_smallest} (d) that for any $\zeta\vdash n$ and $\zeta\ne (n)$, $(n-1,1)$, $(2,1^{n-2})$ or $(1^n)$, we have the smallest eigenvalue $\lambda_{\min}(\rho_\zeta(H))\ge -2(n-2)(n-4)!$, which is strictly larger than $-(n-2)!$. Thus the smallest eigenvalue of $\Gamma$ is $-(n-2)!$, which is attained exactly by $(n-1,1)$ and $(2,1^{n-2})$ with multiplicity $d_{(n-1,1)}\cdot b+d_{(2,1^{n-2})}\cdot b=2(n-1)$.

	Next suppose $2\leq r\leq n-3$.   We find from Lemma~\ref{lem:n-1_even_smallest} (d) that for any $\zeta\vdash n$ and $\zeta\ne (n)$, $(n-1,1)$, $(2,1^{n-2})$ or $(1^n)$, we have the smallest eigenvalue $\lambda_{\min}(\rho_\zeta(H))\ge -2(n-2)(n-4)!>-(n-r)(n-3)!$. Thus the smallest eigenvalue of $\Gamma$ is $-(n-r)(n-3)!$,  attained exactly by $(n-1,1)$ and $(2,1^{n-2})$. The multiplicity of this eigenvalue is $d_{(n-1,1)}\cdot b+d_{(2,1^{n-2})}\cdot b=2(n-1)(r-1)$.

    Finally, suppose $r=n-2$.  We find from Lemma~\ref{lem:n-1_even_smallest} (e) that for any $\zeta\vdash n$ and $\zeta\ne (n)$, $(n-1,1)$, $(2,1^{n-2})$ or $(1^n)$, we have the smallest eigenvalue $\lambda_{\min}(\rho_\zeta(H))\ge -(n-r)(n-3)!$. Thus the smallest eigenvalue of $\Gamma$ is $-(n-r)(n-3)!$, which is attained by $(n-1,1)$ and $(2,1^{n-2})$, and can only be attained possibly by  $(n-1,1)$, $(2,1^{n-2})$, $(n-2,1^2)$, $(3,1^{n-3})$, $(n-2,2)$ and $(2,2,1^{n-4})$.
\end{proof}

\begin{re}
	Corollary~\ref{cor:odd_small} can also be derived based on Lemma~\ref{lem:n-1_even_smallest}. Moreover, when $n=6$, the smallest eigenvalue of $\mathrm{Cay}(S_n,C(n,n-1;n-2))$ is $-12$ with multiplicity $68$, attained exactly by $(n-1,1)$, $(2,1^{n-2})$, $(n-2,1^2)$, $(3,1^{n-3})$, $(n-2,2)$ and $(2,2,1^{n-4})$. 
\end{re}

\section{$\mathbf{C(n,k;r)}$ with $\mathbf{k\leq n-2}$} \label{sec:n-2}

In this section, we prove in Theorems~\ref{thm:even_n-2} and \ref{thm:odd_n-2} that $\mathrm{Cay}(S_n,C(n,k;r))$ has the Aldous property whenever $4\leq k\leq n-2$ and $1\leq r\leq k-1$.  The cases of $k=2$ and $3$ have already been solved in \cite{FOW}  and \cite{HH}, respectively.    Before going straight to proving Theorems~\ref{thm:even_n-2} and \ref{thm:odd_n-2}, we first solve the case of even $k=n-2$, which acts as the base for induction in the proof of Theorem~\ref{thm:even_n-2}. 

\begin{lem}\label{lem:even_k=n-2_r=1_1}
Suppose that $n\ge 8$ is even and $\zeta\vdash n$ such that $\zeta\ne (n)$, $(n-1,1)$, $(2,1^{n-2})$ or $(1^n)$. Let $H=C(n,n-2;1)$. If $\zeta=(n-2,2)$ or $(n-m,1^m)$ with $2\leq m\leq n-3$, then $\lambda_1(\rho_\zeta(H))\leq 0 $; otherwise  $\lambda_1(\rho_\zeta(H))\leq(n-1)(n-4)! $.

\end{lem}
\begin{proof}
Note that $H=C(n,n-2)\setminus (C(n,n-2)\cap G_1)$, where $G_1$ is the stabilizer of $1$ in $S_n$. Similar to the analysis at the start of Section~\ref{sec:proof_for_theorem_n-1}, the map $f_1$ sends $C(n,n-2)\cap G_1$ to $C(n-1,n-2)$ and thus for any $\zeta\vdash n$, the spectrum of $\rho_\zeta(C(n,n-2)\cap G_1)$ is the same as that of $\rho_\zeta(C(n-1,n-2))$.   
By Lemma~\ref{lem:symmetric} and Weyl Inequalities, we deduce 
\begin{align}
   \lambda_1(\rho_\zeta(H))\leq \lambda_1(\rho_\zeta(C(n,n-2)))-\lambda_{\min}(\rho_\zeta(C(n-1,n-2))).\label{ineq:5.1}
\end{align}

The set $C(n,n-2)$ is the conjugacy class of $(n-2)$-cycles in $S_n$ of size $\binom{n}{2}(n-3)!$. Hence by Schur's Lemma, $\rho_\zeta(C(n,n-2))$ is a scalar matrix with the unique eigenvalue $|C(n,n-2)|\cdot \tilde{\chi}_\zeta(\sigma)$, where $\sigma$ is any $(n-2)$-cycle of $S_n$. Then we can refer to Table 4 in \cite{LXZ2} for all nonzero characters of $(n-2)$-cycles in $S_n$. For the reader's convenience, we include this table here.  
\begin{table}[thp]
\renewcommand\arraystretch{2.0}
\centering
\begin{tabular}{|ccc| c c c| c c c|}
\hline
& $\zeta\vdash n$ &&&  $\dim \rho_\zeta$ &&&  $\chi_\zeta(\sigma)$&\\\hline 
& $(n)$ &&& $1$ &&& $1$ &\\ \hline 
& $(1^n)$ &&& $1$ &&& $(-1)^{n-1}$&\\ \hline 
& $(n-1,1)$ &&& $n-1$ &&& $1$ &\\\hline 
& $(2,1^{n-2})$ &&& $n-1$ &&&  $(-1)^{n-1}$& \\\hline
& $(n-2,2)$ &&& $\frac{n(n-3)}{2}$ &&& $-1$& \\ \hline 
& $(2^2,1^{n-4})$ &&& $\frac{n(n-3)}{2}$ &&& $(-1)^{n}$& \\ \hline 
& $(n-m,3,1^{m-3})$ &&& $\frac{n!}{2m(n-2)(n-m)(n-m-1)(m-3)!(n-m-3)!}$ &&& $(-1)^{m}$&\\\hline
& $(n-m,2^2,1^{m-4})$ &&& $\frac{n!}{2(m-1)(m-2)(n-2)(n-m+1)(m-4)!(n-m-2)!}$ &&& $(-1)^{m}$&\\ 
\hline
\end{tabular}
\vspace{0.2cm}
\caption{Nonzero characters of irreducible representations of $S_n$ on an $(n-2)$-cycle $\sigma$}
\label{tab:tab4}
\end{table} 

As for $C(n-1,n-2)$, it is the conjugacy class of $(n-2)$-cycles in $S_{n-1}$. By Branching Rule, for any $\zeta\vdash n$ we have $\rho_\zeta(C(n-1,n-2))=\oplus_{\zeta^-}\rho_{\zeta^-}(C(n-1,n-2))$. Now each $\rho_{\zeta^-}(C(n-1,n-2))$ is a scalar matrix and we can get its unique eigenvalue by Lemma~\ref{lem:n-1_cycles}. 
\begin{itemize}
    \item[$\bullet$]  If $\zeta=(n-2,2)$, then we see from Table~\ref{tab:tab4} that the unique eigenvalue of $\rho_\zeta(C(n,n-2))$ is $-(n-1)(n-4)!$. The eigenvalues of $\rho_{\zeta^-}(C(n-1,n-2))$ are $0$ and $-2(n-3)(n-5)!$ by Lemma~\ref{lem:n-1_cycles}. Thus it follows from \eqref{ineq:5.1} that $\lambda_1(\rho_\zeta(H))\leq -(n-1)(n-4)!+2(n-3)(n-5)!<0$. 
    \item[$\bullet$]  If $\zeta=(2,2,1^{n-4})$, then we have $\lambda_1(\rho_\zeta(C(n,n-2)))=-\lambda_{\min}(\rho_{(n-2,2)}(C(n,n-2))) =(n-1)(n-4)!$, as $\rho_\zeta(C(n,n-2))=-\rho_{(n-2,2)}(C(n,n-2))$ by Lemma~\ref{lem:transpose}. Similarly, the eigenvalues of $\rho_{\zeta^-}(C(n-1,n-2))$ are $0$ and $2(n-3)(n-5)!$. Thus $\lambda_1(\rho_\zeta(H))\leq (n-1)(n-4)!$ follows from \eqref{ineq:5.1}.
    \item[$\bullet$]  If $\zeta=(n-m,2,1^{m-2})$ with $3\leq m\leq n-3$, then we deduce from Table~\ref{tab:tab4} that $\lambda_1(\rho_\zeta(C(n,n-2)))=0$. In this case, $\zeta^-$ has the form $(n-1-s,1^s)$ with $2\leq s\leq n-4$ or $(n-1-t,2,1^{t-2})$ with $2\leq t\leq n-3$. The eigenvalues of $\rho_{\zeta^-}(C(n-1,n-2))$ are $0$ and $(-1)^{t-1}(n-1-t)(n-3-t)!t(t-2)!$ with $2\leq t\leq n-3$. Thus by \eqref{ineq:5.1} we deduce $\lambda_1(\rho_\zeta(H))\leq 2(n-3)(n-5)!$. 
    \item[$\bullet$] If $\zeta=(n-m,3,1^{m-3})$ with $3\leq m\leq n-3$, then $\lambda_1(\rho_\zeta(C(n,n-2)))=(-1)^m m(n-m)(n-m-1)(m-3)!(n-m-3)!$ by Table~\ref{tab:tab4}. 
    	In this case, $\zeta^-=(n-1-m,3,1^{m-3})$ if $n\leq n-4$, $(n-m,3,1^{m-4})$ if $n\ge 4$, or $(n-m,2,1^{m-3})$. By Lemma~\ref{lem:n-1_cycles}, the eigenvalues of $\rho_{\zeta^-}(C(n-1,n-2))$ are $0$ or $(-1)^m(n-m)(n-2-m)!(m-1)(m-3)!$. When $m$ is even,  $\lambda_1(\rho_\zeta(H))\leq m(n-m)(n-m-1)(m-3)!(n-m-3)! \leq 4(n-4)(n-5)(n-7)!$. When $m$ is odd, $\lambda_1(\rho_\zeta(H))\leq - m(n-m)(n-m-1)(m-3)!(n-m-3)!+(n-m)(n-2-m)!(m-1)(m-3)!< 0$.  
     Overall, we have $\lambda_1(\rho_\zeta(H))\leq 4(n-4)(n-5)(n-7)!$.
    \item[$\bullet$] If $\zeta=(n-m,2,2,1^{m-4})$ with $4\leq m\leq n-2$, then we have $\lambda_1(\rho_\zeta(C(n,n-2)))=(-1)^m (m-1)(m-2)(n-m+1)(m-4)!(n-m-2)!$.
      	In this case, $\zeta^-=(n-1-m,2,2,1^{m-4})$ if $m\leq n-3$, $(n-m,2,2,1^{m-5})$ if $m\ge 5$, or $(n-m,2,1^{m-3})$. By Lemma~\ref{lem:n-1_cycles}, the eigenvalues of $\rho_{\zeta^-}(C(n-1,n-2))$ are $0$ or $(-1)^m(n-m)(n-m-2)!(m-1)(m-3)!$. When $m$ is even,  $\lambda_1(\rho_\zeta(H))\leq  (m-1)(m-2)(n-m+1)(m-4)!(n-m-2)!\leq 3(n-3)(n-4)(n-6)!$. When $m$ is odd, $\lambda_1(\rho_\zeta(H))\leq - (m-1)(m-2)(n-m+1)(m-4)!(n-m-2)!+(n-m)(n-m-2)!(m-1)(m-3)!< 0$.  
    Overall, we have $\lambda_1(\rho_\zeta(H))\leq 3(n-3)(n-4)(n-6)!$.

    \item[$\bullet$] For any other $\zeta\vdash n$, Table~\ref{tab:tab4} and Lemma~\ref{lem:n-1_cycles} imply that all the eigenvalues of $\rho_\zeta(C(n,n-2))$ and $\rho_{\zeta^-}(C(n-1,n-2))$ are $0$, and thus $\lambda_1(\rho_\zeta(H))\leq 0$ by \eqref{ineq:5.1}.  

\end{itemize}

To sum up, if $\zeta= (n-2,2)$ or $(n-m,1^m)$ with $2\leq m\leq n-3$, then $\lambda_1(\rho_\zeta(C(n,n-2;1)))\leq 0$; otherwise $\lambda_1(\rho_\zeta(C(n,n-2;1)))\leq (n-1)(n-4)!$.	
\end{proof}

\begin{lem}\label{lem:even_k=n-2_r=1}
Suppose that $n\ge 8$ is an even integer. The second largest eigenvalue of $\Gamma=\mathrm{Cay}(S_n,C(n,n-2;1))$ is attained uniquely by $(n-1,1)$. In particular, $\Gamma$ has the Aldous property.
\end{lem}
\begin{proof}
Let $H=C(n,n-2;1)$. If $\zeta=(n)$, then the unique eigenvalue of $\rho_\zeta(H)$ is $|H|=\binom{n-1}{2}(n-3)!$, which is the largest eigenvalue of $\Gamma$. Similarly, $\zeta=(1^n)$ gives the smallest eigenvalue $-|H|=-\binom{n-1}{2}(n-3)!$ of $\Gamma$. By Lemmas~\ref{lem:transpose} and \ref{lem:eigenvalues_standard_1}, we have $\lambda_1(\rho_{(n-1,1)}(H))=n/2(n-3)!$ and $\lambda_1(\rho_{(2,1^{n-2})}(H))=-\lambda_{\min}(\rho_{(n-1,1)} (H) )= (n-2)!/2$. 
When $\zeta\ne (n)$, $(1^n)$, $(n-1,1)$ or $(2,1^{n-2})$, we have seen from Lemma~\ref{lem:even_k=n-2_r=1_1} that $\lambda_1(\rho_\zeta(C(n,n-2;1)))\leq (n-1)(n-4)!<n/2(n-3)!$. Thus the second largest eigenvalue of $\mathrm{Cay}(S_n, C(n,n-2;1))$ is $n/2(n-3)!$, attained uniquely by $(n-1,1)$.	
\end{proof}

\begin{lem}\label{lem:k=n-2_2}
	Suppose that $n\ge 8$ is even and $2\leq r<(n-1)/2$. The second largest eigenvalue of $\Gamma=\mathrm{Cay}(S_n,C(n,n-2;r))$ is attained uniquely by $(n-1,1)$. In particular, $\Gamma$ has the Aldous property.
\end{lem}

\begin{proof}
      Let $H = C(n,n-2;r) $. If $\zeta=(n)$, then $\lambda_1(\rho_\zeta(H))=|H|=\binom{n-r}{2}(n-3)!$ is the largest eigenvalue of $\Gamma$. If $\zeta=(1^n)$, then $\lambda_1(\rho_\zeta(H))=-|H|=-\binom{n-r}{2}(n-3)!$ is the smallest eigenvalue of $\Gamma$. We then derive from Lemmas~\ref{lem:transpose} and \ref{lem:eigenvalues_standard_2} that \begin{align*}
		\lambda_1(\rho_{(n-1,1)}(H))&=\frac{1}{2}(n-4)!((n-r-1)(n+r-3)+(n-r-3))\\
		&=\frac{1}{2}(n-4)!(n^2-3n-r^2+r)
	\end{align*}
	and
	 \begin{align*}
		\lambda_1(\rho_{(2,1^{n-2})}(H))&=-\lambda_{\min}(\rho_{(n-1,1)}(H))\\
		&= \frac{1}{2}(n-4)!(n-r)(n-r-1)\\
		&=\frac{1}{2}(n-4)!(n^2-2rn-n+r^2+r)\\
		&<\frac{1}{2}(n-4)!(n^2-3n-r^2+r).
	\end{align*}
	Now that we have verified $\lambda_1(\rho_\zeta(H))$ for $\zeta=(n)$, $(1^n)$, $(n-1,1)$ and $(2,1^{n-2})$, we assume in the following that $\zeta\vdash n$ such that $\zeta\ne (n)$, $(1^n)$, $(n-1,1)$ or $(2,1^{n-2})$.

      Notice that $C(n,n-2;i)=C(n,n-2;i-1)\setminus (C(n,n-2;i-1)\cap G_i)$ for any $2\leq i\leq r$. By a similar analysis as at the beginning of Section~\ref{sec:proof_for_theorem_n-1},  we know that the spectrum of $\rho_\zeta(C(n,n-2;i-1)\cap G_i)$ is the same  as that of $\rho_\zeta(C(n-1,n-2;i-1))$ for any $\zeta\vdash n$. Thus by Weyl Inequalities, we have for any $2\leq i\leq r$,
\[ \lambda_1(\rho_\zeta(C(n,n-2;i)))\leq \lambda_1(\rho_\zeta(C(n,n-2;i-1)))-\lambda_{\min}(\rho_\zeta(C(n-1,n-2;i-1))), \]
and thus \begin{align}
	\lambda_1(\rho_\zeta(H))
	&\leq \lambda_1(\rho_\zeta(C(n,n-2;r-1)))-\lambda_{\min}(\rho_\zeta(C(n-1,n-2;r-1))) \nonumber \\ 
    &\leq \lambda_1(\rho_\zeta(C(n,n-2;r-2)))-\sum_{j=r-2}^{r-1}\lambda_{\min}(\rho_\zeta(C(n-1,n-2;j))) \nonumber \\
	&\leq \lambda_1(\rho_\zeta(C(n,n-2;1)))-\sum_{j=1}^{r-1} \lambda_{\min}(\rho_\zeta(C(n-1,n-2;j)))\nonumber.
\end{align}
Branching Rule implies that for  $1\leq j\leq r-1$, \[\rho_\zeta(C(n-1,n-2;j))=\mathop{\oplus}\limits_{\zeta^-} \rho_{\zeta^-}(C(n-1,n-2;j)),\] and so $\lambda_{\min}(\rho_\zeta(C(n-1,n-2;j)))=\min_{\zeta^-} \lambda_{\min}(\rho_{\zeta^-}(C(n-1,n-2;j)))$.
We then have the following inequality:
\begin{align}
	\lambda_1(\rho_\zeta(H))
	\leq \lambda_1(\rho_\zeta(C(n,n-2;1)))-\sum_{j=1}^{r-1} \min_{\zeta^-} \lambda_{\min}(\rho_{\zeta^-}(C(n-1,n-2;j))). \label{ineq:12}
\end{align}
 Now we make use of Lemma~\ref{lem:even_k=n-2_r=1_1} to estimate $\lambda_1(\rho_\zeta(C(n,n-2;1)))$ and Lemma~\ref{lem:n-1_even_smallest} to bound $\lambda_{\min}(\rho_{\zeta^-}(C(n-1,n-2;j)))$ for $1\leq j\leq r-1<(n-3)/2$. 
\begin{itemize}
	\item[$\bullet$] If $\zeta=(n-2,2)$ or $(n-m,1^m)$ with $2\leq m\leq n-3$, then by Lemma~\ref{lem:even_k=n-2_r=1_1}, $\lambda_1(\rho_\zeta(C(n,n-2;1)))\leq 0$. In this case, $\zeta^-=(n-3,2)$ or has the form $(n-1-s,1^s)$ with $1\leq s\leq n-3$. Lemmas~\ref{lem:transpose} and \ref{lem:multiplicity_n-1} give  
            \begin{align*}\label{eq:11}
	         	\lambda_{\min}(\rho_{(n-2,1)}(C(n-1,n-2;j))) 
	         =\begin{cases}
		      -(n-3)!,&\text{ if }j=1;\\
		       -(n-1-j)(n-4)!,& \text{ if }2\leq j\leq  r-1;\end{cases}
	        \end{align*} 
\[\lambda_{\min}(\rho_{(2,1^{n-3})}(C(n-1,n-2;j)))=-\lambda_{1}(\rho_{(n-2,1)}(C(n-1,n-2;j)))=-j(n-4)!.\]

Parts (a) and (b) in Lemma \ref{lem:n-1_even_smallest}  imply that when  $\zeta^-=(n-3,2)$ or $\zeta^-=(n-1-s,1^s)$ with $2\leq s\leq n-4$,
	 \[\lambda_{\min}(\rho_{\zeta^-}(C(n-1,n-2;j)))\ge \begin{cases}
	 	-2(n-3)(n-5)!,&\text{ if }1\leq j\leq 2;\\
	 	-j(n-4)!,&\text{ if }3\leq j\leq r-1.
	 \end{cases} \]

	 Therefore, $\lambda_{\min}(\rho_{(n-2,1)}(C(n-1,n-2;j)))=\min\limits_{\zeta^-} \lambda_{\min}(\rho_{\zeta^-}(C(n-1,n-2;j)))$ for every $j\in \{1,2,\ldots,r-1\}$, and by \eqref{ineq:12} \begin{align*}
	        	\lambda_1(\rho_\zeta(C(n,n-2;r)))&\leq (n-3)! + (n-4)!\sum_{j=2}^{r-1}(n-1-j) \\
	        	&=\frac{1}{2}(n-4)!(2nr-r^2-r-2n )\\
	        	&<\frac{1}{2}(n-4)!(n^2-3n-r^2+r).
	        \end{align*}

	\item[$\bullet$]  For any other  $\zeta$, we know from Lemma~\ref{lem:even_k=n-2_r=1_1} that  $\lambda_1(\rho_\zeta(C(n,n-2;1)))\leq (n-1)(n-4)!$. In this case, $\zeta^-\ne (n-1)$, $(n-2,1)$ or $(1^{n-1})$.  We have seen that $$\lambda_{\min}(\rho_{(2,1^{n-3})}(C(n-1,n-2;j)))=-j(n-4)!.$$ We also know from parts (a) and (b) of Lemma \ref{lem:n-1_even_smallest} that when $\zeta^-\ne (n-1)$, $(n-2,1)$, $(2,1^{n-3})$ or $(1^{n-1})$, 
	 \[\lambda_{\min}(\rho_{\zeta^-}(C(n-1,n-2;j)))\ge \begin{cases}
	 	-2(n-3)(n-5)!,&\text{ if }1\leq j\leq 2;\\
	 	-j(n-4)!,&\text{ if }3\leq j\leq r-1.
	 \end{cases} \]
	 Note that $-2(n-3)(n-5)\leq -j(n-4)!$ for $1\leq j\leq 2$.
    Hence we get by \eqref{ineq:12} \begin{align*}
    	\lambda_1(\rho_\zeta(C(n,n-2;r))) &\leq  (n-1)(n-4)!+4(n-3)(n-5)! + (n-4)!\sum_{j=3}^{r-1}j \\
    	&< \frac{1}{2}(n-4)!(n^2-3n-r^2+r).
    \end{align*}

 \end{itemize}

To sum up, the second largest eigenvalue of $\mathrm{Cay}(S_n,C(n,n-2;r))$ is attained uniquely by $(n-1,1)$.
\end{proof}

 Recall from Lemma~\ref{lem:HHC_upper_bound_1} and its remark  that the eigenvalue 
\begin{align}
	\mu_2(n,k;r)=(k-2) !\binom{n-r}{k-r} \frac{1}{n-r}\left((k-1)(n-k)-\frac{(k-r-1)(k-r)}{n-r-1}\right)\label{eq:mu2}
\end{align}
 	of  $\mathrm{Cay}(S_n,C(n,k;r))$ discovered  in \cite[Theorem 1.3]{SZ} is precisely $\lambda_2(\mathbf{B})$ with $\mathbf{B}$ the quotient matrix of the equitable partition $\Pi_i$ of $\mathrm{Cay}(S_n,C(n,k;r))$, and that $\mu_2(n,k;r)$ is also the largest eigenvalue of the standard representation $\rho_{(n-1,1)}$ of $S_n$ on $C(n,k;r)$, that is, $\alpha_1(\rho_{(n-1,1)}(C(n,k;r)))$. Consequently, $\mu_2(n,k;r)$ provides a lower bound for $\alpha_2(\mathrm{Cay}(S_n,C(n,k;r)))$. One can directly verify from \eqref{eq:mu2} the recurrence relation
\begin{equation}\label{eq:recurrence}
	\mu_2(n,k;r)=\mu_2(n-1,k;r)+\mu_2(n,k;r+1)
\end{equation}
for any $1\leq r\leq k-2$ and $k\leq n-2$.

According to Lemma~\ref{lem:generating_subgroup}, when $k$ is even, $H=C(n,k;r)$ generates $S_n$. In this case, for $r\in\{1,2,\ldots,k-2\}$ and $j\in [n]\setminus [r]$, the three graphs $\mathrm{Cay}(S_n,H)$, $\mathrm{Cay}(S_n,H\setminus( H\cap G_j))$ and $\mathrm{Cay}(G_j,H\cap G_j)$ are all connected, and the strictly second largest eigenvalue of each of them is exactly their second largest eigenvalue. Now we make use of Lemma~\ref{lem:HHC_upper_bound_1} to confirm the Aldous property of  $\mathrm{Cay}(S_n,C(n,k;r))$ when $k$ is even and $k\leq n-2$.

\begin{thm}\label{thm:even_n-2}
	If $k$ is even and $4\leq k\leq n-2$, then $\mathrm{Cay}(S_n,C(n,k;r))$ has the Aldous property for every $r\in \{1,2,\ldots, k-1\}$. 
\end{thm}

\begin{proof}
	Theorem~\ref{thm:induction_base} indicates that $\mathrm{Cay}(S_n,C(n,k;r))$ with $r=k-1$ has the Aldous property, and thus $\lambda_2(\mathrm{Cay}(S_n,C(n,k;k-1)))=\mu_2(n,k;k-1)$. 


	First assume $k=n-2$.  We verify via computation in \textsc{Magma} \cite{BCP} that $(5,1)$ is the unique partition attaining the second largest eigenvalue of $\mathrm{Cay}(S_6,C(6,4;r))$ for $r=1$, $2$ and $3$.   Thus this theorem is true when $k=4$ and $1\leq r\leq 3$.  Lemmas~\ref{lem:even_k=n-2_r=1} and \ref{lem:k=n-2_2} show that when $k\ge 6$ and $1\leq r<(n-1)/2$, the conclusion of this theorem still holds.  Now suppose  $k\ge 6$ and $(n-1)/2\leq r\leq k-1= n-3$, and we prove the theorem by induction on $n-r$. When $n-r=3$, that is, $r=n-3$, we know from Theorem~\ref{thm:induction_base} that $\lambda_2(\mathrm{Cay}(S_n,C(n,n-2;n-3)))=\mu_2(n,n-2;n-3)$.  Suppose $\lambda_2(\mathrm{Cay}(S_n,C(n,n-2;n-t)))=\mu_2(n,n-2;n-t)$ for some $t\in \{3,4,\ldots,(n-2)/2\}$.  Let $n-r=t+1$, that is, $r=n-t-1\in\{n/2, (n+2)/2,\ldots, n-4 \}$, and let $H=C(n,n-2;r)$. By Lemma~\ref{lem:HHC_upper_bound_1}, if $\lambda$ is an eigenvalue of $\mathrm{Cay}(S_n,H)$ other than that of $\mathbf{B}$, then 
	\begin{equation}\label{eq:double_lambda_2}
		\lambda\leq \lambda_2\left(\operatorname{Cay}\left(G_n, H \cap G_n\right)\right)+\lambda_2\left(\operatorname{Cay}\left(S_n, H \backslash\left(H\cap G_n\right)\right)\right).
	\end{equation}
	Here $\mathrm{Cay}(G_n,H\cap G_n)$ is isomorphic to $\mathrm{Cay}(S_{n-1},C(n-1,n-2;r))$ and $\mathrm{Cay}(S_n,H\setminus(H\cap G_n))$ is isomorphic to $\mathrm{Cay}(S_n,C(n,n-2;r+1))$. By the induction hypothesis, we have $\lambda_2(\mathrm{Cay}(S_n,C(n,n-2;r+1)))=\mu_2(n,n-2;r+1)$.  As $(n-1)/2<r\leq n-4$, Theorem~\ref{thm:n-1_odd} (d) indicates that the second largest eigenvalue of $\mathrm{Cay}(S_{n-1},C(n-1,n-2;r))$ is attained uniquely by $(n-1,1)$, and thus $\lambda_2\left(\operatorname{Cay}\left(G_n, H \cap G_n\right)\right)=\mu_2(n-1,n-2;r)$. Then the right hand side of \eqref{eq:double_lambda_2} is  $\mu_2(n-1,n-2;r)+\mu_2(n,n-2;r+1)$, which is exactly $\mu_2(n,n-2;r)$ by the recurrence relation \eqref{eq:recurrence}. Thus \eqref{eq:double_lambda_2} gives $\lambda\leq \mu_2(n,n-2;r)$.  Combining this with \cite[Theorem 1.3]{SZ}, which states that $\mu_2(n,n-2;r)$ is a lower bound for $\lambda_2(\mathrm{Cay}(S_n,C(n,n-2;r)))$, we conclude that $\mu_2(n,n-2;r)$ is the second largest eigenvalue of $\mathrm{Cay}(S_n,C(n,n-2;r))$. Therefore, $\mathrm{Cay}(S_n,C(n,n-2;r))$ has the Aldous property for every $r\in \{1,2,\ldots,n-3\}$.

	Now we conclude the proof by induction on $n-k$. Suppose the conclusion holds when $n-k=i-1$ for some $i\in \{3,4,\ldots,n-4\}$. Next let $n-k=i$. 
	Since Theorem~\ref{thm:induction_base} states  that $\lambda_2(\mathrm{Cay}(S_n,C(n,n-i;n-i-1)))=\mu_2(n,n-i;n-i-1)$, we assume $r\in \{1,2,\ldots, n-i-2\}$ and let $H=C(n,n-i;r)$. By Lemma~\ref{lem:HHC_upper_bound_1}, if $\lambda$ is an eigenvalue of $\mathrm{Cay}(S_n,H)$ other than that of $\mathbf{B}$, then 
	\begin{align*}
		\lambda  \leq \lambda_2(\operatorname{Cay}(S_{n-1}, C(n-1,n-i;r) ))+\lambda_2(\operatorname{Cay}(S_n, C(n,n-i;r+1))). 
	\end{align*}
	As the conclusion of this theorem holds when $n-k=i-1$, we have  $\lambda_2(\mathrm{Cay}(S_{n-1},C(n-1,n-i;r))=\mu_2(n-1,n-i;r)$.   Then based on the recurrence relation \eqref{eq:recurrence} and \cite[Theorem 1.3]{SZ}, the equality $\lambda_2(\operatorname{Cay}(S_n, C(n,n-i;r+1)))=\mu_2(n,n-i;r+1)$ implies  $\lambda_2(\operatorname{Cay}(S_n, H))=\mu_2(n,n-i;r)$.  Hence we conclude that $\mathrm{Cay}(S_n,C(n,n-i;r))$ has the Aldous property for every $r\in \{1,2,\ldots, n-i-1\}$. 

	We finally arrive at that when $k$ is even and $4\leq k\leq n-2$, for any $1\leq r\leq k-1$, the second  largest eigenvalue of $\mathrm{Cay}(S_n,C(n,k;r))$ equals $\mu_2(n,k;r)$, which is exactly $\alpha_1(\rho_{(n-1,1)}(C(n,k;r)))$. In particular, $\mathrm{Cay}(S_n,C(n,k;r))$ has the Aldous property. 
\end{proof}

When $k$ is odd, $H=C(n,k;r)$ only generates $A_n$ according to Lemma~\ref{lem:generating_subgroup}. If $1\leq r\leq k-2$ and $j\in [n]\setminus [r]$, then the three graphs $\mathrm{Cay}(S_n,H)$, $\mathrm{Cay}(S_n,H\setminus( H\cap G_j))$ and $\mathrm{Cay}(G_j,H\cap G_j)$ all have two isomorphic components and their largest eigenvalues all have multiplicity $2$.  In this case, $\mathrm{Cay}(S_n,H)$ has the Aldous property if and only if
\begin{equation}\label{eq:aim_even_n-2}
	\lambda_3(\mathrm{Cay}(S_n,C(n,k;r)))=\mu_2(n,k;r) \quad \text{for}~1\leq r<k\leq n-2.
\end{equation}
 Thus when $k$ is odd, Lemma~\ref{lem:HHC_upper_bound_1} is not strong enough for proving \eqref{eq:aim_even_n-2}. 
Here we have two ways to fix this situation. The first way is to make the alternating group $A_n$ as the underlying group and apply \cite[Theorem 7]{HHC} to $\mathrm{Cay}(A_n,C(n,k;r))$. Since $\lambda_2(\mathrm{Cay}(A_n,C(n,k;r)))=\lambda_3(\mathrm{Cay}(S_n,C(n,k;r)))$, to prove \eqref{eq:aim_even_n-2} we only need to show $\lambda_2(\mathrm{Cay}(A_n,C(n,k;r)))=\mu_2(n,k;r)$.
The second way, which we adopt here, is that with some additional discussion, we can get a stronger version of Lemma~\ref{lem:HHC_upper_bound_1}, which still has $S_n$ as the underlying group. The key point is that when $k$ is odd, all permutations in $H=C(n,k;r)$ are even and thus $\rho_\zeta(H)=\rho_{\zeta'}(H)$ for any $\zeta \vdash n$. From the other point of view, if $\lambda_0$ is an eigenvalue of $\mathrm{Cay}(S_n,H)$ with $f$ an $\lambda_0$-eigenvector, then $g:=f\cdot \mathrm{sgn}$ is also an $\lambda_0$-eigenvector of $\mathrm{Cay}(S_n,H)$. In fact, for any $\sigma\in S_n$,
\begin{align*}
 	Ag(\sigma)&=\sum_{h\in H} g(\sigma h) =\sum_{h\in H} f(\sigma h)\cdot \mathrm{sgn}(\sigma h)= \sum_{h\in H} f(\sigma h)\cdot \mathrm{sgn}(\sigma) \\
 	&=\mathrm{sgn}(\sigma)\cdot \sum_{h\in H} f(\sigma h)=\mathrm{sgn}(\sigma)\cdot \lambda_0 f(\sigma)=\lambda_0g(\sigma),
 \end{align*} 
 where $A$ denotes the adjacency matrix of $\mathrm{Cay}(S_n,H)$. Here we use the facts that $\mathrm{sgn}(h)=1$ for every $h\in H$ and $Af(\sigma)=\sum_{h\in H}f(\sigma h)=\lambda_0 f(\sigma)$ for any $\sigma \in S_n$. In particular, the all one vector $\mathbf{1}$ and the vector $\mathbf{sgn}$ with its $\sigma$-entry $\mathrm{sgn}(\sigma)$ are the two eigenvectors of  the largest eigenvalue $|H|$ of $\mathrm{Cay}(S_n,H)$.  For $i\in [n]$, let $\mathbf{P}_{\Pi_i}$ denote the \emph{characteristic matrix} of the partition $\Pi_i$ of $\mathrm{Cay}(S_n,H)$ given in \eqref{eq:right_cosets_equitable_partition}, which is the $n! \times n$ matrix  with columns the characteristic vectors of $\Pi_i$. If $\lambda$ is an eigenvalue of $\mathrm{Cay}(S_n,H)$ other than that of $\mathbf{B}$ and $f$ is a $\lambda$-eigenvector of $\mathrm{Cay}(S_n,H)$, then $f$ is not only orthogonal to the column space of $\mathbf{P}_{\Pi_i}$ for every $i\in [n]$ but also  orthogonal to the column space of $\mathbf{\tilde{P}}_{\Pi_i}$, which is obtained by multiplying each $(\sigma,G_{j,i})$-entry of $\mathbf{P}_{\Pi_i}$ by $\mathrm{sgn}(\sigma)$. With these additional conditions on $f$, we deduce the following lemma by applying \cite[Theorem 7]{HHC} to $\mathrm{Cay}(S_n,C(n,k;r))$ for odd $k$.

\begin{lem}\label{lem:HHC_upper_bound}
	Let $H=C(n,k;r)$ and $\Gamma=\operatorname{Cay}(S_n, H)$, where $n \ge 5$ and $1 \le r < k < n$. The right coset decomposition $\Pi_i$ of $S_n$ given in \eqref{eq:right_cosets_equitable_partition} leads to an equitable partition of $\Gamma$, and the corresponding quotient matrix $\mathbf{B}=\mathbf{B}_{\Pi_i}$ is symmetric and independent of the choice of $i\in [n]$. Moreover, if $k$ is odd and $\lambda$ is an eigenvalue of $\Gamma$ other than that of $\mathbf{B}$, then, for each $j \in[n]$, we have
$$
\lambda \leqslant \lambda_3\left(\operatorname{Cay}\left(G_j, H \cap G_j\right)\right)+\lambda_3\left(\operatorname{Cay}\left(S_n, H \backslash\left(H\cap G_j\right)\right)\right),
$$
where $G_j$ is the stabilizer of $j$ in $S_n$.
\end{lem}

Using Lemma \ref{lem:HHC_upper_bound}, the proof of the following theorem is similar to that of Theorem~\ref{thm:even_n-2}.

\begin{thm}\label{thm:odd_n-2}
	If $k$ is odd and $5\leq k\leq n-2$, then $\mathrm{Cay}(S_n,C(n,k;r))$ has the Aldous property for every $r\in \{1,2,\ldots,k-1\}$.
\end{thm}

\begin{proof}
Theorem~\ref{thm:induction_base} indicates that $\mathrm{Cay}(S_n,C(n,k;r))$ with $r=k-1$ has the Aldous property, and thus $\lambda_3(\mathrm{Cay}(S_n,C(n,k;k-1)))=\mu_2(n,k;k-1)$.  


	First assume $k=n-2$. We verify via computation in \textsc{Magma} \cite{BCP} that the strictly second largest eigenvalue of $\mathrm{Cay}(S_7,C(7,5;r))$ with $r=1$, $2$, $3$ and $4$ is attained exactly by $(6,1)$ and $(2,1^5)$. Thus this theorem is true when $k=5$ and $1\leq r\leq 4$. Now suppose $k\ge 7$ and $1\leq r\leq k-1= n-3$. We prove the conclusion of this theorem by induction on $n-r$. When $n-r=3$, that is, $r=n-3$, we know from Theorem~\ref{thm:induction_base} that $\lambda_3(\mathrm{Cay}(S_n,C(n,n-2;n-3)))=\mu_2(n,n-2;n-3)$.   Suppose $\lambda_3(\mathrm{Cay}(S_n,C(n,n-2;n-t)))=\mu_2(n,n-2;n-t)$ for some $t\in \{3,4,\ldots, n-2\}$.  Let $n-r=t+1$, that is, $r=n-t-1\in\{ 1,2,\ldots,n-4  \}$, and let $H=C(n,n-2;r)$. By Lemma~\ref{lem:HHC_upper_bound}, if $\lambda$ is an eigenvalue of $\mathrm{Cay}(S_n,H)$ other than that of $\mathbf{B}$, then 
	\begin{equation}\label{eq:double_lambda_3}
		\lambda\leq \lambda_3\left(\operatorname{Cay}\left(G_n, H \cap G_n\right)\right)+\lambda_3\left(\operatorname{Cay}\left(S_n, H \backslash\left(H\cap G_n\right)\right)\right).
	\end{equation}
	Here $\mathrm{Cay}(G_n,H\cap G_n)$ is isomorphic to $\mathrm{Cay}(S_{n-1},C(n-1,n-2;r))$ and $\mathrm{Cay}(S_n,H\setminus(H\cap G_n))$ is isomorphic to $\mathrm{Cay}(S_n,C(n,n-2;r+1))$. By the induction hypothesis, we have $\lambda_3(\mathrm{Cay}(S_n,C(n,n-2;r+1)))=\mu_2(n,n-2;r+1)$.  Theorem~\ref{thm:n-1_even} (b) indicates that the strictly second largest eigenvalue of $\mathrm{Cay}(S_{n-1},C(n-1,n-2;r))$ is attained by $(n-2,1)$, and thus $\lambda_3\left(\operatorname{Cay}\left(G_n, H \cap G_n\right)\right)=\mu_2(n-1,n-2;r)$. Then  the right hand side of \eqref{eq:double_lambda_3} is  $\mu_2(n-1,n-2;r)+\mu_2(n,n-2;r+1)$, which is exactly $\mu_2(n,n-2;r)$ by the recurrence relation \eqref{eq:recurrence}. Thus \eqref{eq:double_lambda_3} gives $\lambda\leq \mu_2(n,n-2;r)$.  Combining this with \cite[Theorem 1.3]{SZ}, which states that $\mu_2(n,n-2;r)$ is a lower bound for $\lambda_3(\mathrm{Cay}(S_n,C(n,n-2;r)))$, we conclude that $\mu_2(n,n-2;r)$ is the strictly second largest eigenvalue of $\mathrm{Cay}(S_n,C(n,n-2;r))$. Therefore, $\mathrm{Cay}(S_n,C(n,n-2;r))$ has the Aldous property for every $r\in \{1,2,\ldots,n-3\}$. 

    Now we conclude the proof by induction on $n-k$. Suppose the conclusion holds when $n-k=i-1$ for some $i\in \{3,4,\ldots,n-4\}$. Next let $n-k=i$. Since Theorem~\ref{thm:induction_base}  states that $\lambda_3(\mathrm{Cay}(S_n,C(n,n-i;n-i-1)))=\mu_2(n,n-i;n-i-1)$, we assume $r\in\{1,2,\ldots,n-i-2\}$ and let $H=C(n,n-i;r)$. By Lemma~\ref{lem:HHC_upper_bound_1}, if $\lambda$ is an eigenvalue of $\mathrm{Cay}(S_n,H)$ other than that of $\mathbf{B}$, then 
	\begin{align*}
		\lambda  \leq \lambda_3(\operatorname{Cay}(S_{n-1}, C(n-1,n-i;r) ))+\lambda_3(\operatorname{Cay}(S_n, C(n,n-i;r+1))). 
	\end{align*}
	As this theorem holds when $n-k=i-1$, we have $\lambda_3(\mathrm{Cay}(S_{n-1},C(n-1,n-i;r))=\mu_2(n-1,n-i;r)$.   Then based on the recurrence relation \eqref{eq:recurrence} and \cite[Theorem 1.3]{SZ},  the equality $\lambda_3(\operatorname{Cay}(S_n, C(n,n-i;r+1)))=\mu_2(n,n-i;r+1)$ implies $\lambda_3(\operatorname{Cay}(S_n, H ))=\mu_2(n,n-i;r)$.  Hence we conclude that $\mathrm{Cay}(S_n,C(n,n-i;r))$ has the Aldous property for every $r\in\{1,2,\ldots, n-i-1\}$. 

	 We finally arrive at that when $k$ is odd and $5\leq k\leq n-2$, for any $1\leq r\leq k-1$, the strictly second largest eigenvalue of $\mathrm{Cay}(S_n,C(n,k;r))$ equals $\mu_2(n,k;r)$, which is exactly $\alpha_1(\rho_{(n-1,1)}(C(n,k;r)))$. In particular, $\mathrm{Cay}(S_n,C(n,k;r))$ has the Aldous property. 
\end{proof}

We finally propose the following conjecture on the strictly second largest eigenvalue of $\mathrm{Cay}(S_n,C(n,k;r))$ with $1\leq r<k\leq n-2$, which is stronger than Theorems~\ref{thm:even_n-2} and \ref{thm:odd_n-2}. 

\begin{cx}
Suppose that $n \ge 5$ and $1\leq r<k\leq n-2$. When $k$ is even, the second largest eigenvalue of $\mathrm{Cay}(S_n,C(n,k;r))$ is attained uniquely by the standard representation $\rho_{(n-1,1)}$. When $k$ is odd, the strictly second largest eigenvalue of $\mathrm{Cay}(S_n,C(n,k;r))$ is attained exactly by  $\rho_{(n-1,1)}$ and $\rho_{(2,1^{n-2})}$.
\end{cx}

The proofs of Theorems~\ref{thm:even_n-2} and \ref{thm:odd_n-2}  rely on Lemmas~\ref{lem:HHC_upper_bound_1} and \ref{lem:HHC_upper_bound}, and we are only able to show that if $\lambda$ is an eigenvalue of $\mathrm{Cay}(S_n,C(n,k;r))$ other than that of $\mathbf{B}$, then  \begin{align}
	 \lambda\leq \mu_2(n,k;r). \label{eq:for cx}
\end{align}  However, the above conjecture requires \eqref{eq:for cx} to be a strict inequality, that is, $\lambda< \mu_2(n,k;r)$. Lemmas~\ref{lem:even_k=n-2_r=1} and \ref{lem:k=n-2_2} state that this conjecture is true for $\mathrm{Cay}(S_n,C(n,n-2;r))$ with $n$ even and $1\leq r\leq (n-1)/2$.


\begin{thebibliography}{}
\bibitem{A} N. Alon. Eigenvalues and expanders. Combinatorica, 6(2):83--96, 1986.

\bibitem{AM} N. Alon and V. D. Milman. $\lambda_1$, isoperimetric inequalities for graphs, and superconcentrators.
J. Combin. Theory Ser. B, 38(1):73--88, 1985.

\bibitem{BCP} W. Bosma, J. Cannon, C. Playoust. The MAGMA algebra system I: the user language. J. Symb. Comput. 24, 235--265 (1997).

\bibitem{BH} A. E. Brouwer and W. H. Haemers, Spectra of Graphs. Springer, Berlin, 2012.

\bibitem{C1} F. Cesi. Cayley graphs on the symmetric group generated by initial reversals have unit spectral gap. Electron. J. Combin., 16(1):N29, 2009.

\bibitem{CLR} P. Caputo, T. M. Liggett and T. Richthammer. Proof of Aldous' spectral gap conjecture.
J. Amer. Math. Soc., 23(3):831--851, 2010.

\bibitem{CT} F. Chung and J. Tobin. The spectral gap of graphs arising from substring reversals.
Electron. J. Combin., 24(3):3--4, 2017.

\bibitem{D} J. Dodziuk. Difference equations, isoperimetric inequality and transience of certain random
walks. Trans. Amer. Math. Soc., 284(2):787--794, 1984.

\bibitem{DS} P. Diaconis and M. Shahshahani. Generating a random permutation with random transpositions. Z. Wahrscheinlichkeitstheor. Verw. Geb., 57(2):159--179, 1981.


\bibitem{FOW} L. Flatto, A. M. Odlyzko, and D. B. Wales, Random shuffles and group representations. Ann. Probab., 13(1):154--178, 1985. 

\bibitem{HH} X. Huang and Q. Huang, The second largest eigenvalues of some Cayley graphs on alternating groups. J. Algebraic Combin., 50(1):99--111, 2019.

\bibitem{HHC} X. Huang, Q. Huang, and S. M. Cioab$\breve{a}$, The second eigenvalue of some normal
Cayley graphs of highly transitive groups. Electron. J. Combin., 26(2):\#P2.44, 28
pages, 2019.

\bibitem{HLW}
S. Hoory, N. Linial and A. Wigderson. Expander graphs and their applications. Bull. Amer. Math. Soc., 43(4):439--561, 2006.

\bibitem{I} I. M. Isaacs. Character Theory of Finite Groups. AMS Chelsea Publishing, Providence, RI, 2006.

\bibitem{JK} G. James and A. Kerber. The Representation Theory of the Symmetric Group. Addison-Wesley, London 1981.

\bibitem{JL} G. James and M. Liebeck. Representations and Characters of Groups. Cambridge Univ. Press, New York, second edition, 2001.

\bibitem{LXZ} Y. Li, B. Xia, and S. Zhou. Aldous' spectral gap property for normal Cayley graphs on symmetric groups. European J. Combin., 110:103657, 2023.

\bibitem{LXZ2} Y. Li, B. Xia, and S. Zhou. The second largest eigenvalue of normal Cayley graphs on symmetric groups generated by cycles. arXiv preprint arXiv:2302.04022, 2023.

\bibitem{LZ}
X. Liu and S. Zhou. Eigenvalues of Cayley graphs. Electron. J. Combin., 29(2):P2.9, 2022.

\bibitem{M} B. Mohar. Isoperimetric numbers of graphs. J. Combin. Theory Ser. B, 47(3):274--291, 1989.

\bibitem{MR} R. Maleki and A. S. Razafimahatratra. On the Second Eigenvalue of Certain Cayley Graphs on the Symmetric Group. Bull. Malaysian Math. Sci. Soc., 46(5):158, 2023.

\bibitem{MS} M. Krebs and A. Shaheen. Expander Families and Cayley Graphs: A Beginner's Guide. Oxford Univ. Press, 2011.


\bibitem{PP} O. Parzanchevski and D. Puder. Aldous's spectral gap conjecture for normal sets, Trans. Amer. Math. Soc. 373(10):7067--7086, 2020.

\bibitem{S1} J. P. Serre. Linear Representations of Finite Groups. Springer, 1977. 

\bibitem{S} B. Sagan, The Symmetric Group: Representations, Combinatorial Algorithms, and Symmetric Functions. Vol. 203,
Springer Science \& Business Media, 2001.

\bibitem{SZ} J. Siemons and A. Zalesski, On the second largest eigenvalue of some Cayley graphs of the symmetric group.  J. Algebraic Combin., 55(3):989--1005, 2022.

\bibitem{W} H. Weyl, Das asymptotische Verteilungsgesetz der Eigenwerte linearer partieller Differentialgleichungen. Math. Ann., 71:441--479, 1912.

\bibitem{Z} P. H. Zieschang. Cayley graphs of finite groups. J. Algebra, 118(2):447--454, 1988.

\end{thebibliography}
\end{document}